\documentclass{article}
\input{macros-arXiv.sty}

\usepackage{url}            
\usepackage{booktabs}       
\usepackage{amsfonts, bm, amssymb}       
\usepackage{nicefrac}       
\usepackage{microtype}      
\usepackage{tikz} 
\usepackage{float}

\usepackage{array}
\usepackage{diagbox}
\usepackage{mathtools}
\usepackage{multirow, makecell}
\usepackage{rotating}
\usepackage{graphicx}
\usepackage{mathdots}

\usepackage{float}
\usepackage{setspace}
\usepackage[letterpaper, left=1in, right=1in, top=1in, bottom=1in]{geometry}

\usepackage[bottom]{footmisc}
\usepackage[colorlinks,linkcolor={blue!75!black},urlcolor=red,citecolor=ForestGreen,backref=page]{hyperref}

 \renewcommand*{\backrefalt}[4]{%
    \ifcase #1%
     \or (cited on page:~#2)%
     \else (cited on pages:~#2)%
    \fi%
    }

\usepackage{threeparttable}

\usepackage{arydshln}

\usepackage{enumerate}
\usepackage{enumitem} 

\usepackage[sort, round]{natbib}

\usepackage{algorithm, algorithmic, setspace}

\renewcommand{\algorithmiccomment}[1]{\bgroup\small\hfill//~#1\egroup}
\usepackage{float}
\floatstyle{plain}
\newfloat{floatalgo}{ht}{floatalgo}

\usepackage{caption}
\usepackage[labelformat=simple]{subcaption}

\usepackage{mathrsfs}

\usepackage{xiangpackage}
\crefname{assume}{assumption}{assumptions}

\title{\bf Nest Your Adaptive Algorithm for Parameter-Agnostic\\ Nonconvex Minimax Optimization}

\author{
Junchi Yang
$^\dagger$
$^*$
\and
Xiang Li
$^\dagger$
$^*$
\and
Niao He
$^\dagger$
}

\date{\vspace{1ex}}

\begin{document}
\maketitle
\def\thefootnote{$*$}\footnotetext{
Equal contribution.
}
\def\thefootnote{$\dagger$}\footnotetext{Department of Computer Science, ETH Zurich, Switzerland.
\texttt{junchi.yang@inf.ethz.ch},
\texttt{xiang.li@inf.ethz.ch},
\texttt{niao.he@inf.ethz.ch}.
}
\def\thefootnote{\arabic{footnote}}

\begin{abstract}
     Adaptive algorithms like AdaGrad and AMSGrad are successful in nonconvex optimization owing to their \textit{parameter-agnostic} ability -- requiring no a priori knowledge about problem-specific parameters nor tuning of learning rates. However, when it comes to nonconvex minimax optimization, direct extensions of such adaptive optimizers without proper \emph{time-scale separation} may fail to work in practice.  We provide such an example proving that the simple combination of Gradient Descent Ascent (GDA) with adaptive stepsizes can diverge if the primal-dual stepsize ratio is not carefully chosen; hence, a fortiori, such adaptive extensions are not parameter-agnostic. To address the issue, we formally introduce a Nested Adaptive framework, NeAda for short, that carries an inner loop for adaptively maximizing the dual variable with controllable stopping criteria and an outer loop for adaptively minimizing the primal variable.
       Such mechanism can be equipped with off-the-shelf adaptive optimizers and automatically balance the progress in the primal and dual variables.
  Theoretically, for nonconvex-strongly-concave minimax problems,
  we show that NeAda with AdaGrad stepsizes can achieve the near-optimal $\widetilde{O}(\epsilon^{-2})$ and $\widetilde{O}(\epsilon^{-4})$ gradient complexities respectively in the deterministic and stochastic settings,  \textit{without} prior information on the problem's smoothness and strong concavity parameters.
    To the best of our knowledge, this is the first algorithm that simultaneously achieves near-optimal convergence rates and parameter-agnostic adaptation in the nonconvex minimax setting. Numerically, we further illustrate the robustness of the NeAda family with  experiments on  simple test functions and a real-world application.

\end{abstract}

\section{Introduction}

Adaptive gradient methods, whose stepsizes and search directions are adjusted based on past gradients, have received phenomenal popularity and are proven successful in a variety of large-scale machine learning applications. Prominent examples include AdaGrad~\citep{duchi2011adaptive}, RMSProp~\citep{hinton2012neural}, AdaDelta~\citep{zeiler2012adadelta}, Adam~\citep{kingma2015adam}, and AMSGrad~\citep{reddi2018convergence}, just to name a few. Their empirical success is especially pronounced for nonconvex optimization such as training deep neural networks.  Besides improved performance, being \textit{parameter-agnostic} is another important trait of adaptive methods. Unlike (stochastic) gradient descent,  adaptive methods often do not require a priori knowledge about problem-specific parameters (such as Lipschitz constants, smoothness, etc.).\footnote{For distinction, we use ``parameter-agnostic'' to describe algorithms that do not ask for problem-specific parameters in setting their stepsizes or hyperparameters;  we refer to "adaptive algorithms" as methods whose stepsizes are based on the previously observed gradients.  }   On the theoretical front, some adaptive methods can achieve nearly the same convergence guarantees as (stochastic) gradient descent~\citep{duchi2011adaptive, ward2019adagrad,reddi2018convergence}.

Recently, adaptive methods have sprung up for minimax optimization:
\begin{equation}\label{eq:main_problems}
  \min_{x\in \mathbb{R}^{d}}\max_{y\in \mathcal{Y}}\ f(x, y) \triangleq \mathbb{E}[F(x, y ; \xi)],
\end{equation}
where $f$ is $l$-Lipschitz smooth jointly in $x$ and $y$, $\mathcal{Y}$ is closed and convex, and $\xi$ is a random vector. Such problems have found numerous applications in generative adversarial networks (GANs)~\citep{goodfellow2014generative,arjovsky2017wasserstein}, Wasserstein GANs~\citep{arjovsky2017wasserstein}, generative adversarial imitation learning \citep{ho2016generative}, reinforcement learning \citep{dai2017learning, modi2021model}, adversarial training~\citep{tramer2018ensemble}, domain-adversarial training of neural networks~\citep{ganin2016domain} , etc.

A common practice is to simply combine adaptive stepsizes with popular minimax optimization algorithms such as Gradient Descent Ascent (GDA), extragradient method (EG) and the like; see e.g., \citep{gidel2018variational, gulrajani2017improved,goodfellow2016nips}. It is worth noting that  these methods are reported successful in some applications yet at other times can suffer from training instability.   In recent years,  theoretical behaviors of such adaptive methods are extensively studied for convex-concave minimax optimization; see e.g., ~\citep{bach2019universal, antonakopoulos2019adaptive, antonakopoulos2021adaptive2, ene2020adaptive, stonyakin2018generalized, gasnikov2019adaptive, malitsky2020golden, diakonikolas2020halpern}. However, for minimax optimization in the important nonconvex regime, little theory related to adaptive methods is known.  

Unlike the convex-concave setting, a key challenge for nonconvex minimax optimization lies in the necessity of a \emph{problem-specific time-scale separation} of the learning rates between the min-player and max-player when GDA or EG methods are applied, as proven in~\cite{yang2021faster,lin2020gradient, sebbouh2022randomized, boct2020alternating}.  This makes the design of adaptive methods fundamentally different from and more challenging than nonconvex minimization. Several recent attempts~\citep{guo2021novel,huang2021adagda,huang2021efficient} studied adaptive methods for nonconvex-strongly-concave minimax problems;  yet,  they all require explicit knowledge of the problems' smoothness and strong concavity parameters  to maintain a stepsize ratio proportional to the condition number. Such a requirement evidently undermines the parameter-agnostic trait of adaptive methods. This then raises a couple of interesting questions: (1) \emph{Without a problem-dependent stepsize ratio, does simple combination of GDA and adaptive stepsizes still converge?} (2) \emph{Can we design an adaptive algorithm for nonconvex minimax optimization that is truly parameter-agnostic and provably convergent?}

\begin{figure}[t]
    \centering
    \includegraphics[width=\textwidth]{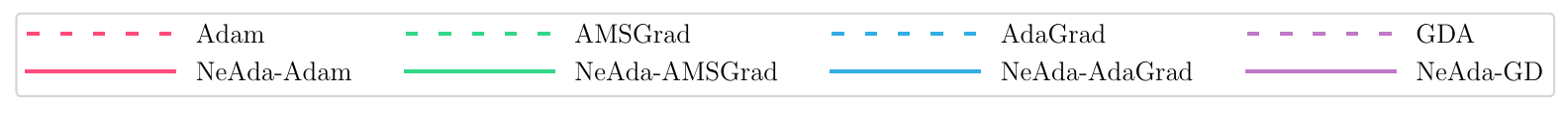}
    \begin{subfigure}[b]{0.25\textwidth}
      \centering
      \includegraphics[width=\textwidth]{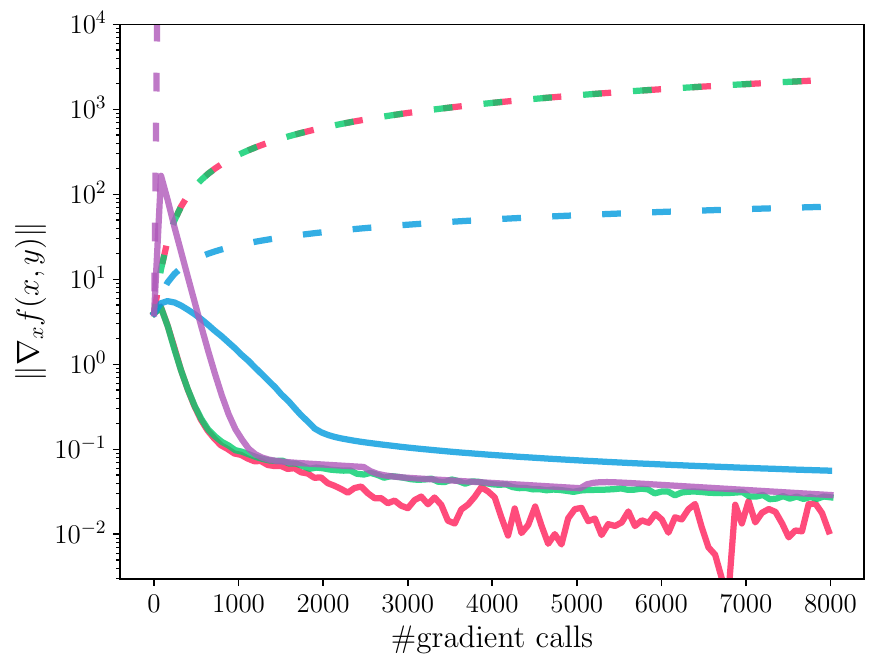}
      \caption{$r = 1$}
    \end{subfigure}
    \begin{subfigure}[b]{0.24\textwidth}
      \centering
      \includegraphics[width=\textwidth]{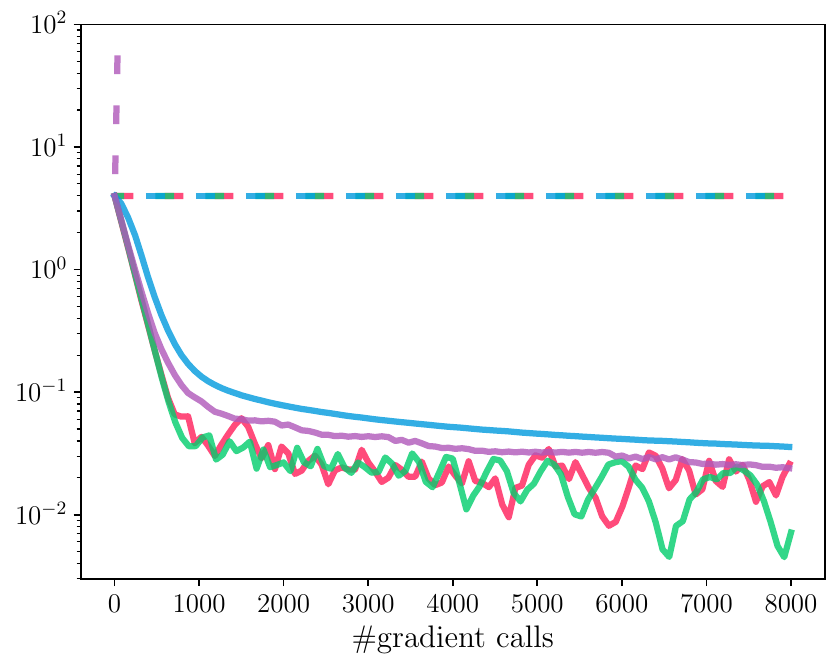}
      \caption{$r = 2$}
    \end{subfigure}
    \begin{subfigure}[b]{0.24\textwidth}
      \centering
      \includegraphics[width=\textwidth]{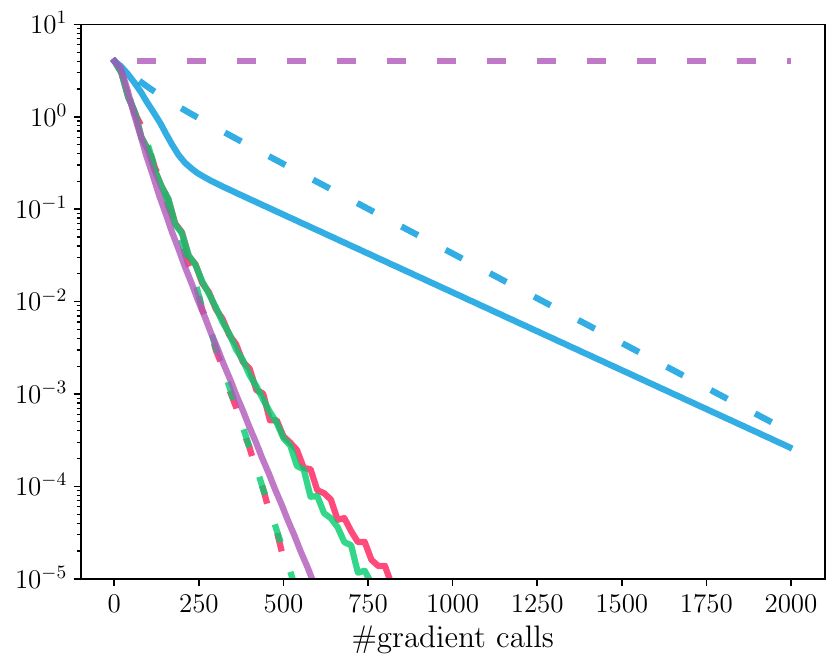}
      \caption{$r = 4 $}
    \end{subfigure}
    \begin{subfigure}[b]{0.24\textwidth}
      \centering
      \includegraphics[width=\textwidth]{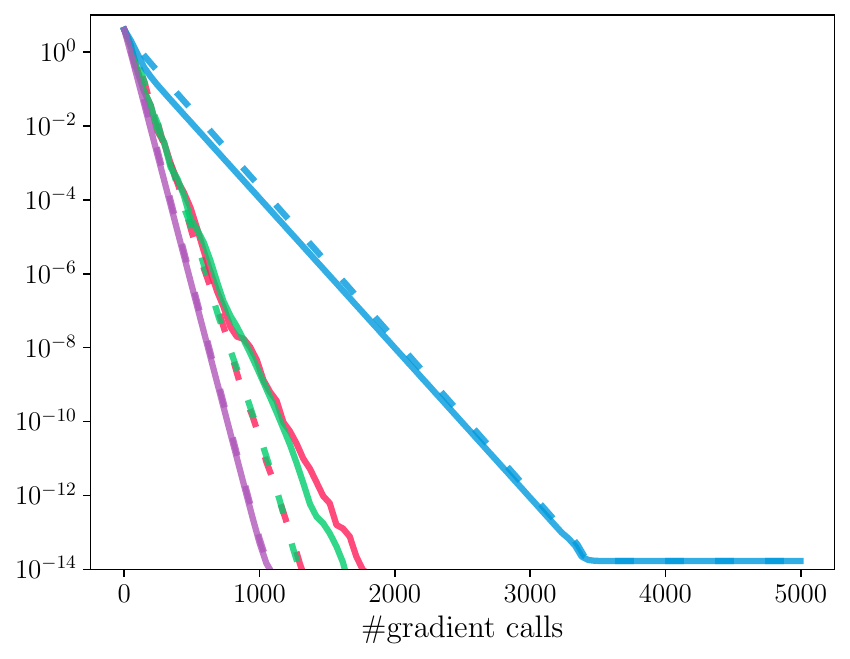}
      \caption{$r = 8 $}
    \end{subfigure}
    \caption{Comparison between the two families of non-nested and nested adaptive methods  on function $f(x,y)=-\frac{1}{2}y^2 + 2xy - 2x^2$
    with deterministic gradient oracles. $r=\eta^y/\eta^x$ is a pre-fixed learning rate ratio.}%
    \label{fig:1_2_func_determ_diff_ratio}
\end{figure}

 In this paper, we address these questions and make the following key contributions:
\begin{itemize}
    \item We investigate two generic frameworks for adaptive minimax optimization: one is a  simple (non-nested) adaptive framework, which performs one step of update of $x$ and $y$ simultaneously with adaptive gradients; the other is Nested Adaptive (NeAda) framework, which performs multiple updates of $y$  after one update of $x$, each with adaptive gradients. Both frameworks allow flexible choices of adaptive mechanisms such as Adam, AMSGrad and AdaGrad. We provide an example proving that the simple adaptive framework can fail to converge without setting an appropriate stepsize ratio; this applies to any of the adaptive mechanisms mentioned above, even in the noiseless setting. In contrast, the NeAda framework is less sensitive to the stepsize ratio, as numerically illustrated in \Cref{fig:1_2_func_determ_diff_ratio}. 

    \item We provide the convergence analysis for a representative of NeAda that uses  AdaGrad stepsizes for $x$ and a convergent adaptive optimizer for $y$, in terms of nonconvex-strongly-concave minimax problems. Notably, the convergence of this general scheme does not require to know any problem parameters and does not assume the bounded gradients.     We demonstrate that NeAda is able to achieve $\widetilde{O}(\epsilon^{-2})$ oracle complexity for the deterministic setting and $\widetilde{O}(\epsilon^{-4})$ for the stochastic setting to converge to $\epsilon$-stationary point, matching best known bounds. To the best of our knowledge, this seems to be the first adaptive framework
    for nonconvex minimax optimization that is provably convergent and parameter-agnostic.
    
    \item We further make two complementary contributions, which can be of independent interest.
    First, we propose a general AdaGrad-type stepsize for strongly-convex problems without knowing the strong convexity parameters, and derive a convergence rate comparable to SGD. It can serve as a subroutine for NeAda.  Second, we provide a high probability convergence result for the primal variable of NeAda under a subGaussian assumption.
    
    \item Finally, we numerically validate the robustness of the NeAda framework on several test functions compared to the non-nested adaptive framework, and demonstrate the effectiveness of the NeAda framework on distributionally robust optimization task with a real dataset. 
    
\end{itemize}

\subsection{Related work}

\paragraph{Adaptive algorithms. }
\citet{duchi2011adaptive} introduce AdaGrad for convex online learning and achieve $O(\sqrt{T})$ regrets.  \citet{li2019convergence} and \citet{ward2019adagrad} show an $\widetilde{O}(\epsilon^{-4})$ complexity for AdaGrad in the nonconvex stochastic optimization. There are an extensive number of works on AdaGrad-type methods; to list a few, \citep{levy2018online, antonakopoulos2021adaptive, kavis2019unixgrad, orabona2018scale}. Another family of algorithms uses more aggressive stepsizes of exponential moving average of the past gradients, such as Adam~\citep{kingma2015adam} and RMSProp~\citep{hinton2012neural}. \citet{reddi2018convergence} point out the non-convergence of Adam and provide a remedy with non-increasing stepsizes. There is a surge in the study of Adam-type algorithms due to their popularity in the deep neural network training~\citep{zaheer2018adaptive, chen2018convergence, liu2019variance}. Some work provides the convergence results for adaptive methods in the strongly-convex optimization~\citep{wang2019sadam, levy2017online, mukkamala2017variants}. Line search and stochastic line search are another effective strategy that can detect the objective's curvature and have received much attention~\citep{vaswani2019painless, vaswani2021towards, vaswani2020adaptive}. Notably, many adaptive algorithms are parameter-agnostic~\citep{duchi2011adaptive, reddi2018convergence, ward2019adagrad}.

\paragraph{Nonconvex minimax optimization. }
Stationary convergence of GDA in NC-SC setting was first provided by \citet{lin2020gradient}, showing $O(\epsilon^{-2})$ oracle complexity and $O(\epsilon^{-4})$ sample complexity with minibatch. Recently, \citet{chen2021closing} and~\citet{yang2021faster} achieve this sample complexity in the stochastic setting without minibatch. GDmax is a double loop algorithm that maximizes the dual variable to a certain accuracy. It achieves nearly the same complexity as GDA~\citep{nouiehed2019solving}. \citet{sebbouh2022randomized} recently discuss the relation between the two-time-scale and number of inner steps for GDmax. Very recently, \citet{li2022convergence} provide the necessary and sufficient time-scale separation for GDA to converge locally to Stackelberg equilibrium. Besides NC-SC setting, some work provides convergent algorithms when the objective is (non-strongly) concave about the dual variable\citep{zhang2020single, lu2020hybrid, yang2020catalyst}. Nonconvex-nonconcave regime is only explored under some special structure~\citep{liu2021first, diakonikolas2020halpern}, such as Polyak-\L{}ojasiewicz (PL) condition~\citep{fiez2021global, yang2020global}. All algorithms mentioned above require prior knowledge about problem parameters, such as smoothness modulus, strong concavity modulus, and noise variance.

\paragraph{Adaptive algorithms in minimax optimization. }
There exist many  adaptive and parameter-agnostic methods designed for convex-concave minimax optimization as a special case of monotone variational inequality~\citep{bach2019universal, antonakopoulos2019adaptive, antonakopoulos2021adaptive2, ene2020adaptive, stonyakin2018generalized, gasnikov2019adaptive, malitsky2020golden, diakonikolas2020halpern}. Most of them combine extragradient method,  mirror prox~\citep{nemirovski2004prox} or the like, with AdaGrad mechanism. \citet{liu2019towards} and \citet{dou2021one} relax convexity-concavity assumption to the regime where  Minty variational inequality (MVI) has a solution. In these settings,  time-scale separation of learning rates is not required even for non-adaptive algorithms. For nonconvex-strongly-concave problems, \citet{huang2021adagda, huang2021efficient, guo2021novel} propose adaptive methods, which set the learning rates based on knowledge about smoothness and strong-concavity modulus and the bounds for adaptive stepsizes. 

\section{Non-nested and nested adaptive methods} \label{sec::nonconvergence}

In this section, we investigate two generic frameworks that can incorporate most existing adaptive methods into minimax optimization. We remark that many variants  encapsulated in these two families are already widely used in practice, such as training of GAN~\citep{goodfellow2016nips}, distributionally robust optimization~\citep{DBLP:conf/iclr/SinhaND18}, etc. These two frameworks, coined as non-nested and nested adaptive methods, can be viewed as adaptive counterparts of GDA and GDmax. We aim to illustrate the difference between these two adaptive families, even though GDA and GDmax are often considered ``twins''.

\paragraph{Non-nested adaptive methods. } In Algorithm \ref{alg::adagda}, non-nested methods update the primal and dual variables in a symmetric way. Weighted gradients $m_t^x$ and $m_t^y$ are the moving average of the past stochastic gradients with the momentum parameters $\beta^x$ and $\beta^y$. The effective stepsizes of $x$ and $y$ are $\eta^x/\sqrt{v_t^x}$ and $\eta^y/\sqrt{v_t^y}$, where the division is taken coordinate-wise. We refer to $\eta^x$ and $\eta^y$ as learning rates, and  $v_t^x, v_t^y$ are some average of squared-past gradients through function $\psi$. Many popular choices of adaptive stepsizes are captured in this framework, see also \citep{reddi2018convergence}:

 \begingroup
 \setlength{\abovedisplayskip}{-10pt}
 \setlength{\belowdisplayskip}{-10pt}
 \setlength{\abovedisplayshortskip}{-10pt}
 \setlength{\belowdisplayshortskip}{-10pt}
 \addtolength{\jot}{-0.7em}
 \begin{align*}
       &\text{(GDA)} \quad \beta = 0; \  \ \psi\left(v_0, \{g^2_i\}_{i=0}^{t}\right) = 1, \quad
       \text{(AdaGrad)} \quad \beta = 0; \  \ \psi\left(v_0, \{g^2_i\}_{i=0}^{t}\right) = v_0 + \sum_{i=0}^{t} g_i^2, \\
       &\text{(Adam)} \quad \psi\left(v_0, \{g^2_i\}_{i=0}^{t}\right) = \gamma^{t+1}v_0 + (1-\gamma)\sum_{i=0}^t\gamma^{t-i}g_i^2, \\
       &\text{(AMSGrad)} \quad \psi\left(v_0, \{g^2_i\}_{i=0}^{t}\right) = \max_{m=0,\dots, t} \gamma^{m+1}v_0 + (1-\gamma)\sum_{i=0}^m\gamma^{m-i}g_i^2. \\
 \end{align*}
 \endgroup

\begin{floatalgo}
\begin{minipage}[H]{0.46\textwidth}
\begin{algorithm}[H] 
\setstretch{1.27}
    \caption{Non-nested Adaptive Method}
    \begin{algorithmic}[1]
        \STATE Input: $x_0$ and $y_0$
        \FOR{$t = 0,1,2,...$}
            \STATE sample $\xi_t$ and let \\ $g_t^x = \nabla_x F(x_t, y_t; \xi_t)$ and \\ $g_t^y = \nabla_y F(x_t, y_t; \xi_t)$ 
            \STATE \texttt{// update the first moment}
            \\ $m_{t+1}^x = \beta^xm^x_{t} + (1-\beta^x)g_t^x$ and \\ $ m_{t+1}^y = \beta^y m^y_{t} + (1-\beta^y)g_t^y$
            \STATE \texttt{// update the second moment}
            \\ $v_{t+1}^x = \psi\left(v_0^x, \{(g_i^x)^2\}_{i=0}^{t}\right)$ and \\ $v_{t+1}^y = \psi\left(v_0^y, \{(g_i^y)^2\}_{i=0}^{t}\right)$
            \STATE \texttt{// update variables}
            \\ $x_{t+1} = x_t - \frac{\eta^x}{\sqrt{v_{t+1}^x}} m_{t+1}^x$ and \\ $ y_{t+1} = y_t + \frac{\eta^y}{\sqrt{v_{t+1}^y}} m_{t+1}^y$
        \ENDFOR
    \end{algorithmic} \label{alg::adagda}
\end{algorithm}
\end{minipage}
\hfill
\begin{minipage}[H]{0.52\textwidth}
\begin{algorithm}[H] 
\setstretch{1.28}
    \caption{Nested Adaptive (NeAda) Method}
    \begin{algorithmic}[1]
        \STATE Input: $x_0$ and $y_0^0$
        \FOR{$t = 0,1,2,...$}
        \FOR{$k = 0,1,2,...$ until a stopping criterion is satisfied}
            \STATE sample $\hat{\xi}_t^k$ and  $g_{t,k}^y = \nabla_y F(x_t, y_t^k; \hat{\xi}_t^k)$
            \STATE $ m_{t,k+1}^y = \beta^y m^y_{t,k} + (1-\beta^y)g_{t,k}^y$
            \STATE $ v_{t,k+1}^y = \psi^y\left(v_{t,0}^y, \{(g_{t, i}^y)^2\}_{i=0}^{k} \right)$
            \STATE $ y_{t}^{k+1} = y_t^k + \frac{\eta^y}{\sqrt{v_{t,k+1}^y}} m_{t,k+1}^y$
          \ENDFOR
            \STATE $v_{t+1, 0}^y = v_{t,k+1}^y$ and $m_{t+1, 0}^y = m^y_{t,k+1}$ 
            \STATE sample $\xi_t$ and  $g_t^x = \nabla_x F(x_t, y_t^{k+1}; \xi_t)$
            \STATE $m_{t+1}^x = \beta^x m^x_{t} + (1-\beta^x)g_t^x$ 
            \STATE $v_{t+1}^x = \psi^x\left(v_0^x, \{(g_i^x)^2\}_{i=0}^{t}\right)$
            \STATE $x_{t+1} = x_t - \frac{\eta^x}{\sqrt{v_{t+1}^x}} m_{t+1}^x$
        \ENDFOR
    \end{algorithmic} \label{alg::neada}
\end{algorithm}
\end{minipage}
\end{floatalgo}

\paragraph{Nested adaptive (NeAda) methods. } NeAda, presented
in Algorithm \ref{alg::neada}, has a nesting inner loop to maximize $y$ until some stopping criterion is reached (see details in Section~\ref{sec::convergence}).   Instead of using a fixed number of inner iterations or a fixed target accuracy as in GDmax~\citep{lin2020gradient, nouiehed2019solving}, NeAda gradually increases the accuracy of the inner loop as the outer loop proceeds to make it fully adaptive.

We refer to the ratio between two learning rates, i.e. $\eta^y/\eta^x$, as the two-time-scale. The current analysis of GDA in nonconvex-strongly-concave setting requires two-time-scale to be proportional with the condition number $\kappa = l/\mu$, where $l$ and $\mu$ are Lipschitz smoothness and strongly-concavity modulus~\citep{lin2020gradient, yang2021faster}. We provide an example showing that the problem-dependent two-time-scale is \textit{necessary} for GDA and most non-nested methods even in the deterministic setting.

\begin{lemma}
\label{lemma:nonconverge}
Consider the function  $f(x, y) = -\frac{1}{2}y^2 + Lxy - \frac{L^2}{2}x^2$ in the deterministic setting. Let $r \eta^x = \eta^y$. (1) GDA will not converge to the stationary point when $ r \leq L^2$:
\begin{equation*}
    \nabla_x f(x_T, y_T) = \nabla_x f(x_0, y_0) \prod_{t=0}^{T-1}\left[ 1 + \eta^x(L^2 - r) \right].
\end{equation*}
(2) Assume the averaging function $\psi^x$ and $\psi^y$ are the same, and satisfy that for any $\tau$, if $v_t^x = \tau v_t^y$ and $(g_t^x)^2 = \tau (g_t^y)^2$ then $v_{t+1}^x = \tau v_{t+1}^y$.  With $\beta^x = \beta^y$,  $v_0^x = v_0^y = 0$ and $m_0^x = m_0^y = 0$ (which are commonly used in practice), non-nested adaptive method will not converge when $r \leq L$:
\begin{equation*}
    \nabla_x f(x_T, y_T) \geq \nabla_x f(x_0, y_0) \prod_{t=0}^{T-1}\left[ 1 + \frac{L\eta^x}{\sqrt{v_t^x}}(1 - \beta^x)(L - r) \right].
\end{equation*}
When $r = L$, $\nabla_x f(x_t, y_t) = \nabla_x f(x_0, y_0)$ for all $t$.
\end{lemma}

\begin{remark}
Most popular adaptive stepsizes we mentioned before, such as Adam, AMSGrad and AdaGrad, have averaging functions satisfying the assumption in the lemma. 
Any point on the line $y = Lx$ is a stationary point for the above function,  and the distance from a point to this line is proportional to its gradient norm, so the divergence in gradient norm will also implies that of iterates. In the proof, we will also show that the averaged or best iterate will still diverge under the same condition.  
The lemma implies that for any given time-scale $r$, there exists a problem for which the non-nested algorithm does not converge to the stationary point, so they are not parameter-agnostic. 
\end{remark}

We compare non-nested and nested methods combined with different stepsizes schemes: Adam, AMSGrad, AdaGrad and fixed stepsize, on the function: $-\frac{1}{2}y^2 + 2xy - 2x^2$. In the experiments of this section, we halt the inner loop when the (stochastic) gradient about $y$ is smaller than $1/t$ or the number iteration is greater than $t$. We observe from Figure \ref{fig:1_2_func_determ_diff_ratio} that the thresholds for the non-convergence of non-nested methods ($r=2$ for adaptive methods and $r=4$ for GDA) are exactly as predicted by the lemma. Although the adaptive methods admit a smaller two-time-scale threshold than GDA in this example, it is not a universal phenomenon from our experiments in Section \ref{sec:experiments}. Interestingly, nested adaptive methods are robust to different two-time-scales and always have the trend to converge to the stationary point.

\section{Convergence Analysis of NeAda-AdaGrad} \label{sec::convergence}

In this section, we reveal the secret behind the robust performance of NeAda by providing the
convergence guarantee for a representative member in the family. For sake of simplicity and clarity, we mainly focus on NeAda with AdaGrad.  Adam-type mechanism can suffer from non-convergence already for nonconvex minimization despite its good performance in practice. Our result also sheds light on the analysis of other more sophisticated members such as AMSGrad in the family. 

\paragraph{NeAda-AdaGrad:} Presented in \Cref{alg::neada-adagrad}, NeAda-AdaGrad adopts the scalar AdaGrad scheme for the $x$-update in the outer loop and uses mini-batch in the stochastic setting. For the inner loop for maximizing $y$, we run some adaptive algorithm for maximizing $y$ until some easily checkable stopping criterion is satisfied. We suggest two criteria here: at $t$-th outer loop: (I) the squard gradient mapping norm about $y$ is smaller than $1/(t+1)$ in the deterministic setting, (II) the number of inner loop iterations reaches $t + 1$ in the stochastic setting.

\begin{algorithm}[ht] 
    \caption{NeAda-AdaGrad}
    \setstretch{1.25}
    \begin{algorithmic}[1]
        \STATE Input: $(x_0, y_{-1})$, $v_0 > 0$, $\eta > 0$.
        \FOR{$t = 0,1,2,..., T-1$}
            \STATE from $y_{t-1}$ run an adaptive algorithm $\mathcal{A}$ for maximizing $f(x_t, \cdot)$ to obtain $y_{t}$ \\
            (a) stopping criterion I (deterministic): stop when $\|y_t - \text{Proj}_{\mathcal{Y}}(y_t + \nabla_y f(x_t, y_t))\|^2 \leq \frac{1}{t+1}$  \\
            (b) stopping criterion II (stochastic): stop after $t+1$ inner loop iterations.
            \vspace{1mm}
            \STATE $v_{t+1} = v_{t} + \left\|\frac{1}{M}\sum_{i=1}^{M}\nabla_x F(x_t, y_t; \xi^i_t)\right\|^2$ where $\{ \xi^i_t\}_{i=1}^{M}$ are i.i.d samples \label{b sample}
            \STATE $x_{t+1} = x_t - \frac{\eta}{\sqrt{v_{t+1}}} \left( \frac{1}{M}\sum_{i=1}^{M}\nabla_x F(x_t, y_t; \xi^i_t)\right)$ 
            \label{gradient sample}
        \ENDFOR
    \end{algorithmic} \label{alg::neada-adagrad}
\end{algorithm}

For the purpose of theoretical analysis, we mainly focus on the minimax problem of the form (\ref{eq:main_problems}) under the nonconvex-strongly-concave (NC-SC) setting\footnote{Note that for other nonconvex minimax optimization beyond the NC-SC setting,  even the convergence of non-adaptive gradient methods has not been fully understood.}, formally stated in the following assumptions. 

\begin{assume} [Lipschitz smoothness]
There exists a positive constant $l>0$ such that
{\small
\begin{align*}
     \max\big\{\left\| \nabla _ { x } f \left( x _ { 1 } , y _ {1} \right) - \nabla _ { x } f \left( x _ { 2 } , y _ {2} \right) \right\|, \left\| \nabla _ { y } f \left( x _{1} , y _ { 1 } \right) - \nabla _ { y } f \left( x_ {2} , y _ { 2 } \right) \right\| \big \} \leq l [\left\| x _ { 1 } - x _ { 2 } \right\| + \left\| y _ { 1 } - y _ { 2 } \right\|],
\end{align*} \label{assum:smooth}
}%
holds for all $x_1, x_2\in \mathbb{R}^{d}, y_1$, $y_2 \in \mathcal{Y}$.
\end{assume}

\begin{assume} [Strong-concavity in $y$]
\label{assum:sc}
 There exists $\mu > 0$ such that: $
  f(x, y_1) \geq f(x,y_2) + \langle \nabla_y f(x, y_1), y_1 - y_2\rangle + \frac{\mu}{2}\|y_1 - y_2\|^2, \forall x \in \mathbb{R}^d, y_1, y_2 \in \mathcal{Y}.
$
\end{assume}

For simplicity of notation, define $\kappa = l/\mu$ as the condition number, $\Phi(x) = \max_{y \in \mathcal{Y}} f(x, y)$ as the primal function, and $y^*(x) = \argmax_{\mathcal{Y}}f(x, y)$ as the optimal $y$ w.r.t $x$. Since the objective is nonconvex about $x$, we aim at finding an $\epsilon$-stationary point $(x_t, y_t)$ such that $\mathbb{E}\|\nabla_x f(x_t, y_t)\| \leq \epsilon$ and $\mathbb{E}\|y_t - y^*(x_t)\| \leq \epsilon$, where the expectation is taken over the randomness in the algorithm.

\subsection{Convergence in deterministic and stochastic settings}

\begin{assume}[Stochastic gradients]
\label{assum:stoc_grad}  
$\nabla_x F(x,y; \xi)$ and $\nabla_y F(x,y; \xi)$ are unbiased stochastic estimators of $\nabla_x f(x, y)$ and $\nabla_y f(x, y)$ and have variances bounded by $\sigma^2 \geq 0$.  
\end{assume}

We assume the unbiased stochastic gradients have the variance $\sigma^2$, and the problem reduces to the deterministic setting when $\sigma = 0$. Now we provide a general analysis of the convergence for any adaptive optimizer used in the inner loop. 
\begin{theorem} 
\label{lemma stoc}
 Define the expected cumulative suboptimality of inner loops as $\mathcal{E} = \mathbb{E}\left[ \sum_{t=0}^{T-1} \frac{l^2\|y_t - y^*(x_t)\|^2 }{2\sqrt{v_0}}\right]$. Under Assumptions \ref{assum:smooth}, \ref{assum:sc} and \ref{assum:stoc_grad}, the output from Algorithm \ref{alg::neada-adagrad} satisfies 
\begin{equation*}
    \mathbb{E}\left[\sqrt{\frac{1}{T}\sum_{t=0}^{T-1} \|\nabla_x f(x_t, y_t)\|^2} \right] \leq \frac{2(A + \mathcal{E})}{\sqrt{T}} + \frac{v_0^{\frac{1}{4}}\sqrt{A + \mathcal{E}}}{\sqrt{T}} + \frac{2\sqrt{(A + \mathcal{E})\sigma}}{(M T)^{\frac{1}{4}}},
\end{equation*}
where 
$A = \frac{2\Delta}{\eta} + \left( \frac{4\sigma}{\sqrt{M }} + 2\kappa l \eta\right) \left[ 1 + 2\log \left( \mathrm{Poly} \left(T, \mathcal{E}, \frac{\Delta}{\eta}, \frac{\sigma}{\sqrt{M }}, \kappa l\eta, v_0, \frac{1}{v_0} \right) \right) \right]$.
\end{theorem}

\begin{remark}
  The general analysis is built upon milder assumptions than existing work on AdaGrad in nonconvex optimization, not requiring either bounded gradient in~\citep{ward2019adagrad} or prior knowledge about the smoothness modulus in~\citep{li2019convergence}. This theorem implies the algorithm attains convergence for the nonconvex variable $x$ with any constant $\eta > 0$ and $v_0 > 0$ that does not depend on any problem parameter, so it is parameter-agnostic. 
\end{remark}

\begin{remark}
Another benefit of this analysis is that the variance $\sigma$ appears in the leading term $T^{-\frac{1}{4}}$, which means the convergence rate can interpolate between the deterministic and stochastic settings. It implies a complexity of $\widetilde{O}(\epsilon^{-2})$ in the deterministic setting and $\widetilde{O}(\epsilon^{-4})$ in the stochastic setting for the primal variable as long as the accumulated suboptimality for the inner-loops $\mathcal{E}$ is $\widetilde{\mathcal{O}}(1)$, regardless of the batch size $M$.
However, $M$ can control the number of outer loops and there affect the sample complexity for the dual variable. 
\end{remark}

In the next two theorems,  we derive the total complexities, in the deterministic and stochastic settings, of finding $\epsilon$-stationary point by controlling the cumulative suboptimality $\mathcal{E}$ in \Cref{lemma stoc} for subroutine $\mathcal{A}$ with specific convergence rate.   In fact, we can also use any off-the-shelf adaptive optimizer for solving the inner maximization problem up to the desired accuracy. 
Note that (stochastic) GDmax fixes each inner-loop's accuracy or steps to be related with $\mu$, $\ell$ and $\epsilon$ so that $\mathcal{E}$ can be easily bounded \citep{lin2020gradient, nouiehed2019solving}. In contrast, since we do not have access to the problem parameters and $\epsilon$, Algorithm \ref{alg::neada-adagrad} gradually increases the inner-loop accuracy. In the proof of the following theorems, we will show that with our proposed stopping criteria and desired subroutines, $\mathcal{E}$ is bounded by $\mathcal{O}(\log T)$.

\begin{theorem} [deterministic] 
\label{thm:deter}
Suppose we have a linearly-convergent subroutine $\mathcal{A}$ for maximizing any strongly concave function $h(\cdot)$:
$$\|y^k - y^*\|^2 \leq a_1(1-a_2)^k\|y^0 - y^*\|^2$$
where $y^k$ is $k$-th iterate, $y^*$ is the optimal solution, and $a_1>0$ and $0 < a_2 <1$ are constants that can depend on the parameters of $h$. 
Under the same setting as  \Cref{lemma stoc} with $\sigma = 0$,  
for Algorithm~\ref{alg::neada-adagrad} with  $M = 1$ and a subroutine $\mathcal{A}$ under stopping criterion I,  there exists $t^* \leq  \widetilde{O}\left(\epsilon^{-2} \right)$ such that 
$(x_{t^*}, y_{t^*})$ is an $\epsilon$-stationary point.
Therefore, the total gradient complexity is $\widetilde{O}\left( \epsilon^{-2} \right)$.
\end{theorem}

\begin{remark}
  This complexity is optimal in $\epsilon$ up to logarithmic term~\citep{zhang2021complexity}, similar to GDA~\citep{lin2020gradient}. Note that many adaptive and parameter-agnostic algorithms  can achieve the linear rate when solving smooth and strongly concave maximization problems; to list a few, gradient ascent with backtracking line-search~\citep{vaswani2019painless}, SC-AdaNGD~\citep{levy2017online} 
  and polyak stepsize~\citep{hazan2019revisiting,loizou2021stochastic,orvieto2022dynamics} \footnote{\citet{levy2017online} needs to know the diameter of $\mathcal{Y}$. 
  \citet{hazan2019revisiting,loizou2021stochastic,orvieto2022dynamics} use polyak stepsize which
  requires knowledge of the minimum or lower bound of the function value.
   AdaGrad achieves the linear  rate if the learning rate is smaller than $O(1/l)$, and $O(1/k)$ rate otherwise \citep{xie2020linear}.
  }. Here we can also pick more general subproblem accuracy in criterion I that only needs to scale with $1/t$. 
\end{remark}

\begin{theorem} [stochastic] 
\label{thm:stoc}
Suppose we have a sub-linearly-convergent subroutine $\mathcal{A}$ for maximizing any strongly concave function $h(\cdot)$: after $K = k+1$ iterations
$$   \mathbb{E} \|y^K - y^*\|^2 \leq \frac{b_1\|y^0 - y^*\|^2 + b_2}{k},$$
where $y^k$ is $k$-th iterate, $y^*$ is the optimal solution, and $b_1, b_2>0$ are constants that can depend on the parameters of $h$. 
Under the same setting as  \Cref{lemma stoc},  for Algorithm \ref{alg::neada-adagrad} with $M = \epsilon^{-2}$ and subroutine $\mathcal{A}$ under the stopping criterion II,  there exists $t^* \leq \widetilde{O}\left(\epsilon^{-2} \right)$ such that 
$(x_{t^*}, y_{t^*})$ is an $\epsilon$-stationary point.
Therefore, the total stochastic gradient complexity is 
$\widetilde{O}\left(\epsilon^{-4} \right).$
\end{theorem}

\begin{remark}
  This $\widetilde{O}\left(\epsilon^{-4} \right)$ complexity is nearly optimal in the dependence of $\epsilon$ for stochastic NC-SC problems~\citep{li2021complexity}. Here we set $M = \epsilon^{-2}$ for the simplicity of exposition, and a similar result also holds for gradually increasing $M$. The sublinear rate specified above for solving the stochastic strongly convex subproblem can be achieved by several existing parameter-agnostic algorithms under some additional assumptions, such as
  \textsc{FreeRexMomentum}~\citep{NIPS2017_6aed000a} 
  and Coin-Betting~\citep{cutkosky2018black}\footnote{
\textsc{FreeRexMomentum}~\citep{NIPS2017_6aed000a} and Coin-Betting~\citep{cutkosky2018black} can achieves $\mathcal{O}(\log k/k)$ convergence rate when the stochastic gradient is bounded in $\mathcal{Y}$. If the subroutine has additional logarithmic dependence, it suffices to run the subroutine for $t\log^2(t)$ times using criterion II (see Appendix \ref{apdx:sec3}).}.
Parameter-free SGD~\citep{carmon2022making}
is partially parameter-agnostic that only requires the stochastic gradient bound rather than the strongly-convexity parameter.
\citet{mukkamala2017variants} and \citet{wang2019sadam} introduce the variants of AdaGrad, RMSProp and Adam for strongly-convex online learning, but they need to know both gradient bounds and strongly-convexity parameter for setting stepsizes. We will show in the next subsection that AdaGrad with a slower decaying rate is parameter-agnostic. We note that the analysis of this theorem is not the simple gluing of the outer loop and inner loop complexity, but requires more sophisticated control of the cumulative suboptimality $\mathcal{E}$. 
\end{remark}

With the popularity of computational resource demanding deep neural networks, in both minimization and minimax applications, people find high probability guarantees for a single  run of an algorithm useful~\citep{kavis2022high,DBLP:journals/corr/abs-2007-14294}. Given the lack of such guarantee in the minimax optimization, we provide a high probability convergence result for NeAda-AdaGrad in \Cref{apdx:high_prob}, which shows a similar sample complexity as  \Cref{thm:stoc} under the subGaussian noise.

\subsection{Generalized AdaGrad for strongly-convex subproblem}
We now introduce the generalized AdaGrad for minimizing strongly convex objectives, which can serve as an adaptive subroutine for \Cref{alg::neada-adagrad}, without requiring knowledge on the strongly convex parameter. We analyze it for the more general online convex optimization setting: at each round $t$, the learner updates its decision $x_t$, then it suffers a loss $f_t(x_t)$ and receives the sub-gradient of $f_t$. The generalized AdaGrad, described in \Cref{alg:adagrad_sc}, keeps the cumulative gradient norm $v_t$ and takes the stepsize $\eta/v_t^\alpha$ with a decaying rate $\alpha \in (0,1]$. When $\alpha = 1/2$, it reduces to the scalar version of the original AdaGrad \citep{duchi2011adaptive}; when $\alpha = 1$, it reduced to the scalar version of SC-AdaGrad \citep{mukkamala2017variants}.

\begin{algorithm}[ht] 
    \caption{Generalized AdaGrad for Strongly-convex Online Learning}
    \setstretch{1.25}
    \begin{algorithmic}[1]
        \STATE Input: $x_0$, $v_0 > 0$ and $0 < \alpha \leq 1$ .
        \FOR{$t = 0,1,2,...$}
            \STATE receive $g_t \in \partial f_t(x_t)$
            \STATE $v_{t+1} = v_{t} + \| g_t \|^2$ 
            \STATE $x_{t+1} = \mathcal{P}_{\mathcal{X}} \left(x_t - \frac{\eta}{v_{t+1}^{\alpha}} g_t\right)$ 
        \ENDFOR
    \end{algorithmic} \label{alg:adagrad_sc}
\end{algorithm}

\begin{theorem}
\label{thm:gen_ada}
 Consider \Cref{alg:adagrad_sc} for online convex optimization and assume that (i) $f_t$ is continuous and $\mu$-strongly convex,  (ii) $\mathcal{X}$ is convex and compact with diameter $\mathcal{D}$;
 (iii) $\|g_t\| \leq G$ for every $t$. Then for $0 < \alpha < 1$ with any $\eta > 0$, the regret of \Cref{alg:adagrad_sc} satisfies:
  \begin{equation*}
      \max_{x \in \mathcal{X}} \sum_{t = 0}^{T-1} (f_t(x_t) - f_t(x)) \leq  
      c_{\alpha} + d_{\alpha} \left(v_0 + \sum_{t = 1}^{T-1} \|g_t\|^2  \right)^{1-\alpha},   
\end{equation*}
    and for $\alpha = 1$ with $\eta \geq \frac{G^2}{2 \mu}$, 
\begin{equation*}
    \max_{x \in \mathcal{X}} \sum_{t = 0}^{T-1} (f_t(x_t) - f_t(x)) \leq  
    c_{\alpha} + d_{\alpha} \log\left(v_0 + \sum_{t = 1}^{T-1} \|g_t\|^2  \right),  
\end{equation*} 
  where $c_\alpha$ and $d_\alpha$ are constants depending on the problem parameters, $\alpha$ and $\eta$. 
\end{theorem}

  The theorem implies a logarithmic regret for the case $\alpha = 1$, but the stepsize needs knowledge about problem's parameters $\mu$ and $G$; similar results are  shown for SC-AdaGrad~\citep{mukkamala2017variants} and SAdam~\citep{wang2019sadam}. When $\alpha < 1$, the algorithm becomes parameter-agnostic and attains an $O(T^{1-\alpha})$ regret. Such parameter-agnostic phenomenon for smaller decaying rates is also observed for SGD in stochastic optimization \citep{fontaine2021convergence}.  Proving the regret bound for the generalized AdaGrad with $\alpha < 1$ in the online setting is  challenging, since the adversarial $g_t$ can lead to a ``sudden'' change in the stepsize. In the proof, we bound the possible number of times such ``sudden'' change could happen. 
  
  To the best of our knowledge, this is the first regret bound for adaptive methods with general decaying rates in the strongly convex setting. By online-to-batch conversion~\citep{kakade2008generalization},  it can be converted to $O(T^{-\alpha})$ rate in the strongly convex stochastic optimization. \citet{xie2020linear} prove the $O(1/T)$ convergence rate, or a linear convergence rate when the smoothness parameter is known, for AdaGrad with $\alpha = 1/2$ in this setting, but under a strong assumption --- Restricted Uniform Inequality of Gradients (RUIG) --- that requires the loss function with respect to each sample $\xi$ to satisfy the error bound condition with some probability.

\section{Experiments}
\label{sec:experiments}

To evaluate the performance of NeAda,
we conducted experiments on simple test functions and
a real-world application of distributional robustness optimization (DRO). 
In all cases, we compare NeAda with the non-nested adaptive methods 
using  the same adaptive schemes. 
For notational simplicity, in all figure legends, we label the non-nested methods with the
names of the adaptive mechanisms used.
We observe from all our experiments that: 1) while non-nested adaptive
methods can diverge without the proper two-time-scale, NeAda with adaptive subroutine always converges; 2)
when the non-nested method converges, NeAda can achieve comparable or even better performance.

\begin{figure}[t]
    \centering
    \includegraphics[width=\textwidth]{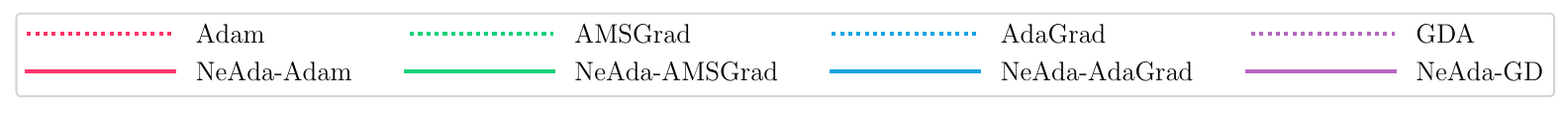}
    \begin{subfigure}[b]{0.25\textwidth}
      \centering
      \includegraphics[width=\textwidth]{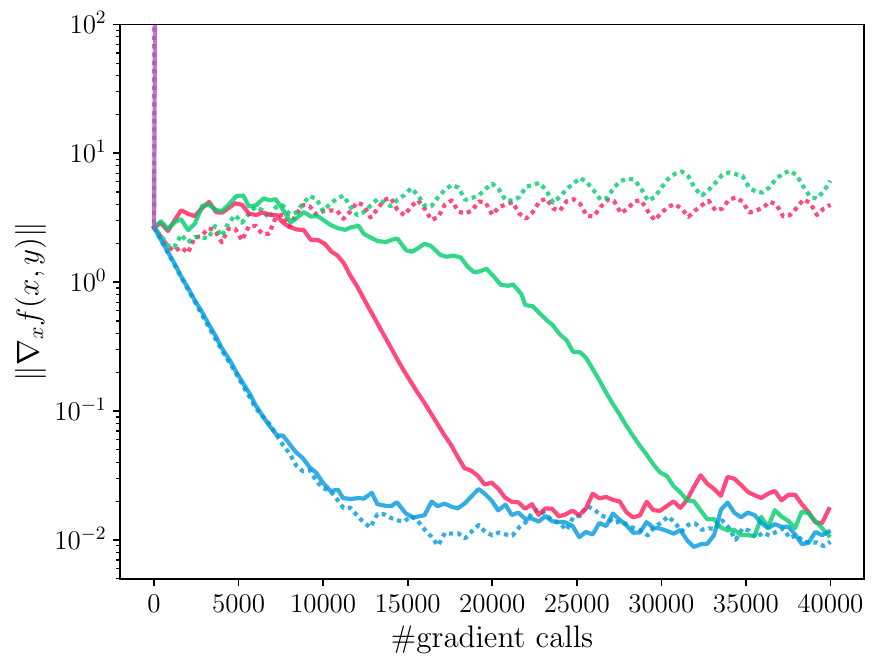}
      \caption{$r = 0.01$}
    \end{subfigure}
    \hfill
    \begin{subfigure}[b]{0.24\textwidth}
      \centering
      \includegraphics[width=\textwidth]{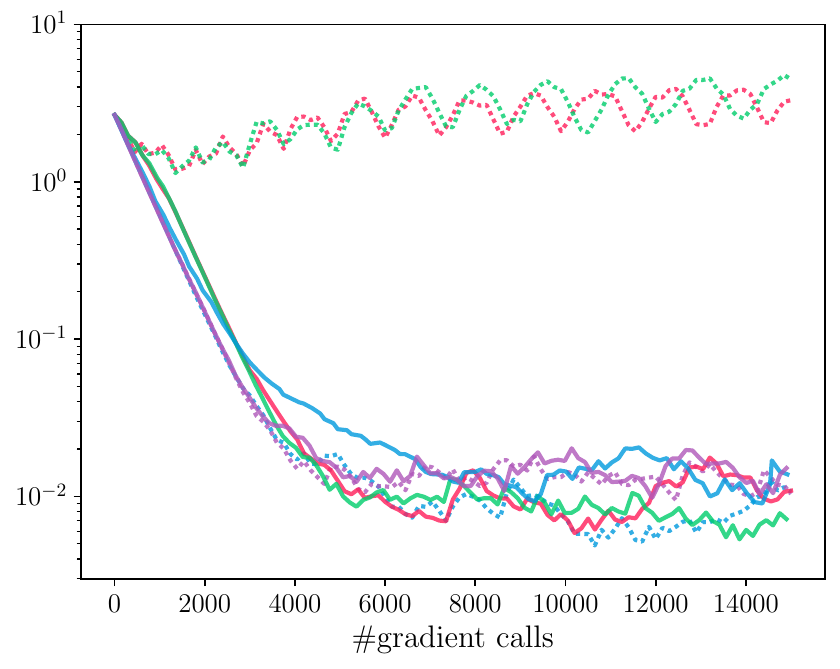}
      \caption{$r = 0.03$}
    \end{subfigure}
    \hfill
    \begin{subfigure}[b]{0.24\textwidth}
      \centering
      \includegraphics[width=\textwidth]{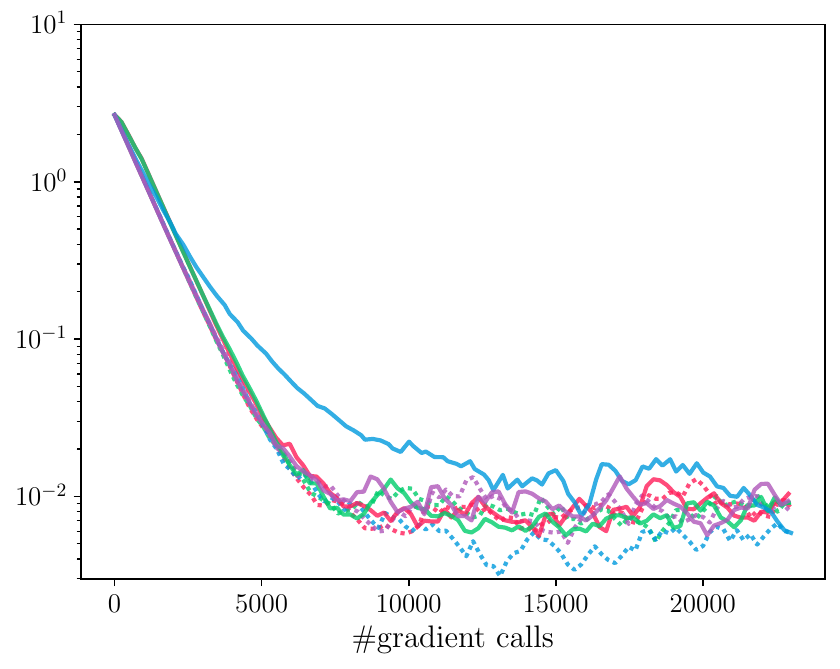}
      \caption{$r = 0.05$}
    \end{subfigure}
    \caption{Comparison between the two families of non-nested and nested adaptive methods on McCromick function
    with stochastic gradient oracles. $\sigma=0.01$, $\eta^y = 0.01$ and $r = \eta^y / \eta^x$.}%
    \label{fig:mccormick_stoch_diff_ratio}
\end{figure}

\subsection{Test functions} 

In Section \ref{sec::nonconvergence}, we have compared NeAda with non-nested methods on a quadratic function in \Cref{fig:1_2_func_determ_diff_ratio} and the observations match \Cref{lemma:nonconverge}. 
Now we consider a more complicated function
that is composed of McCormick function in $x$, a bilinear term, and a quadratic term in $y$, 
\[
f(x, y) = \sin(x_1 + x_2) + (x_1 - x_2)^2 - \frac{3}{2}x_1 + \frac{5}{2}x_2 + 1 + x_1 y_1 + x_2 y_2 - \frac{1}{2}(y_1^2 + y_2^2),
\]
For this function, we compare the adaptive frameworks in the stochastic setting with Gaussian noise.
As demonstrated in \Cref{fig:mccormick_stoch_diff_ratio},
non-nested methods are sensitive to the selection of the two-time-scale.
When the learning rate ratio is too small, e.g., $\eta^y/\eta^x = 0.01$, non-nested Adam, AMSGrad and GDA all fail to converge. We observe that GDA converges when the ratio reaches 0.03, while non-nested Adam and
AMSGrad still diverge until 0.05. Although non-nested adaptive methods require a smaller ratio than GDA in Lemma \ref{lemma:nonconverge}, this example illustrates that adaptive algorithms sometimes can be more sensitive to the time separation. In comparison, NeAda with adaptive
subroutine always converges regardless of the learning rate ratio.

\subsection{Distributional robustness optimization}

To justify the effectiveness
of NeAda on real-world applications, we carried out experiments on 
distributionally robust optimization~\citep{DBLP:conf/iclr/SinhaND18}, where
the primal variable is the model weights to be learned by minimizing the 
empirical loss while the dual variable is the adversarial perturbed
inputs. The dual variable problem targets finding perturbations that
maximize the empirical loss but not far away from the original inputs.
Formally, for model weights $x$ and adversarial samples $y$, we have:
\[
    \min_x \max_{y = [y_1,\dots,y_n]} f(x, y), \quad \text{where} \quad
f(x, y) \coloneqq \frac{1}{n} \sum_{i=1}^n f_i(x, y_i) - \gamma \norm*{y_i - v_i}^2,
\]
where $n$ is the total number of training samples, $v_i$ is the $i$-th original
input and $f_i$ is the loss function for the $i$-th sample.
$\gamma$ is a trade-off parameter between the empirical loss and the
magnitude of perturbations. When $\gamma$ is large enough, this problem
is nonconvex-strongly-concave, and following the same setting as
\citep{DBLP:conf/iclr/SinhaND18, sebbouh2022randomized}, we set $\gamma=1.3$. 
For NeAda, we use both
stopping criterion I with stochastic gradient and criterion II in our experiments.
For the results, we report the training loss and
the test accuracy on adversarial samples generated from
fast gradient sign method~(FGSM)~\citep{goodfellow2014explaining}.
FGSM can be
formulated as
\[
  x_{\text{adv}} = x + \epsilon \cdot \text{sign}\left(\nabla_x f(x)\right),
\]
where $\epsilon$ is the noise level. 
To get reasonable test accuracy,
NeAda with Adam as subroutine is compared with Adam with fixed 15 inner loop
iterations, which is consistent with the choice of inner loop steps in
\citep{DBLP:conf/iclr/SinhaND18}, and such choice obtains much better test
accuracy than the completely non-nested Adam.
Our experiments include a synthetic dataset and MNIST~\citep{lecun1998mnist}
with code modified from \citep{githubcode}.

\begin{figure}[t]
    \centering
    \begin{subfigure}[b]{0.4\textwidth}
      \centering
      \includegraphics[width=\textwidth]{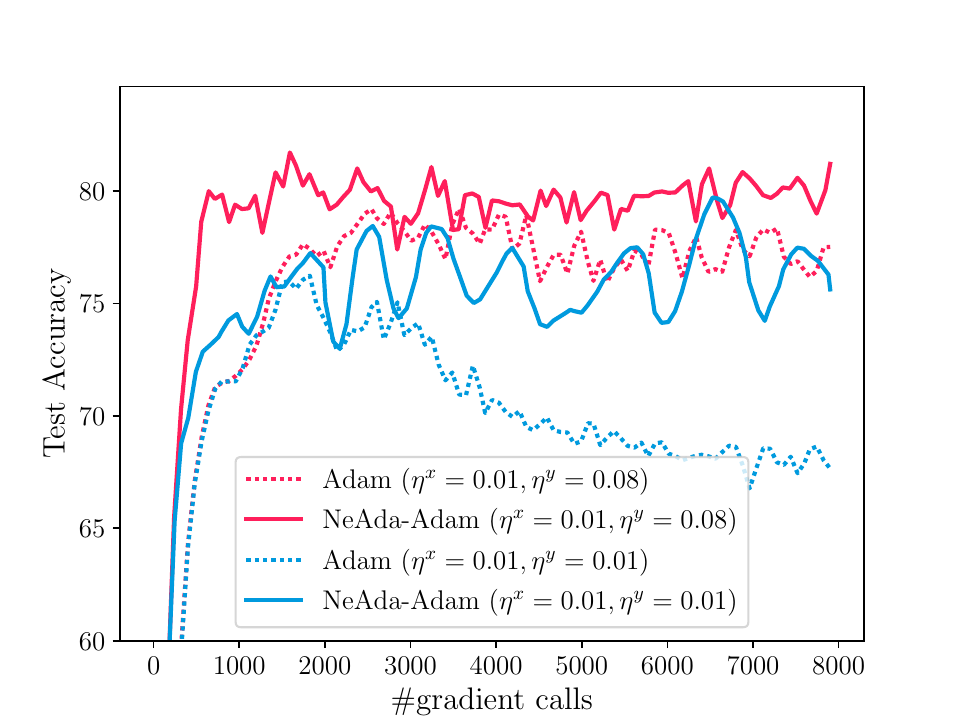}
      \caption{Test Accuracy}
      \label{fig:robust_optim_acc_2}
    \end{subfigure}
    \hspace{0.7cm}
    \begin{subfigure}[b]{0.4\textwidth}
      \centering
      \includegraphics[width=\textwidth]{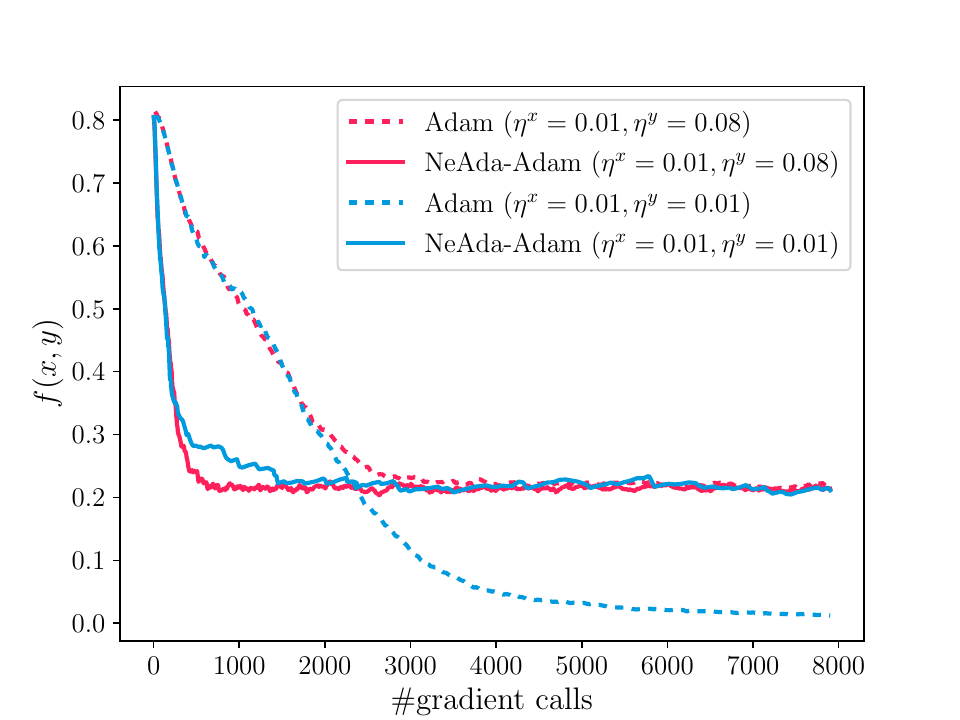}
      \caption{Training Loss}
      \label{fig:robust_optim_loss_2}
    \end{subfigure}
    \caption{Experimental results of distributional robustness optimization task on synthetic dataset.}%
    \label{fig:robust_optim_2}
\end{figure}

\paragraph{Results on Synthetic Dataset.}
We use the same data generation process as in~\citep{DBLP:conf/iclr/SinhaND18}. The inputs are 2-dimensional i.i.d. random Guassian vectors, i.e., $x_i \sim \mathcal{N}(0, I_2)$, where $I_2$ is the $2\times 2$ identity matrix. The corresponding $y_i$ is defined as
$y_i=\text{sign}(\norm*{x_i}_2 - \sqrt{2})$. Data points with norm in range $(\sqrt{2}/1.3, 1.3\sqrt{2})$ are removed to make the classification margin wide.
$10000$ training and $4000$ test data points are generated for our experiments.
The model we use is a three-layer MLP with ELU activations.

As shown in \Cref{fig:robust_optim_acc_2}, when the learning rates
are set to different scales, i.e., $\eta^x = 0.01, \eta^y = 0.08$ (red curves
in the figure), both methods achieve reasonable test errors. In this case,
NeAda has higher test accuracy and reaches such accuracy faster than Adam.
If we change the learning rates to the same scale, i.e., 
$\eta^x = 0.01, \eta^y=0.01$ (blue curves in the figure), NeAda retains
good accuracy while Adam drops to an unsatisfactory performance.
This demonstrates the adaptivity and less-sensitivity to learning
rates of NeAda. In addition, \Cref{fig:robust_optim_loss_2} illustrates the 
convergence speeds on the loss function, and NeAda (solid lines) always
decreases the loss faster than Adam. Note that Adam with the same
learning rates converges to a lower loss but suffers from overfitting, as shown
in \Cref{fig:robust_optim_acc_2} that its test accuracy is only about 68\%.

\paragraph{Results on  MNIST Dataset.}
For MNIST, we use a convolutional neural network with three convolutional
layers and one final
fully-connected layer. Following each convolutional layer,
ELU activation and batch normalization are used.

We compare NeAda with Adam under three
different noise levels and the accuracy is shown in 
\Cref{fig:mnist_result_0.1,fig:mnist_result_0.05,fig:mnist_result_0.02}.
Under all noise levels, NeAda outperforms Adam with the same learning rates.
When we have proper time-scale separation (the red curves), both methods
achieve good test accuracy, and NeAda achieves higher accuracy and converges faster. After we change to
the same learning rates for the primal and dual variables (the blue curves),
the accuracy drop of NeAda is slighter compared to Adam, especially when
$\epsilon=0.1$. As for the training loss shown in \Cref{fig:mnist_loss}, NeAda 
(the solid curves) is always faster at the beginning.
We also observed that with proper time-scale separation, NeAda reaches a lower loss.

\begin{figure}[t]
    \centering
    \begin{subfigure}[b]{0.245\textwidth}
      \centering
      \includegraphics[width=\textwidth]{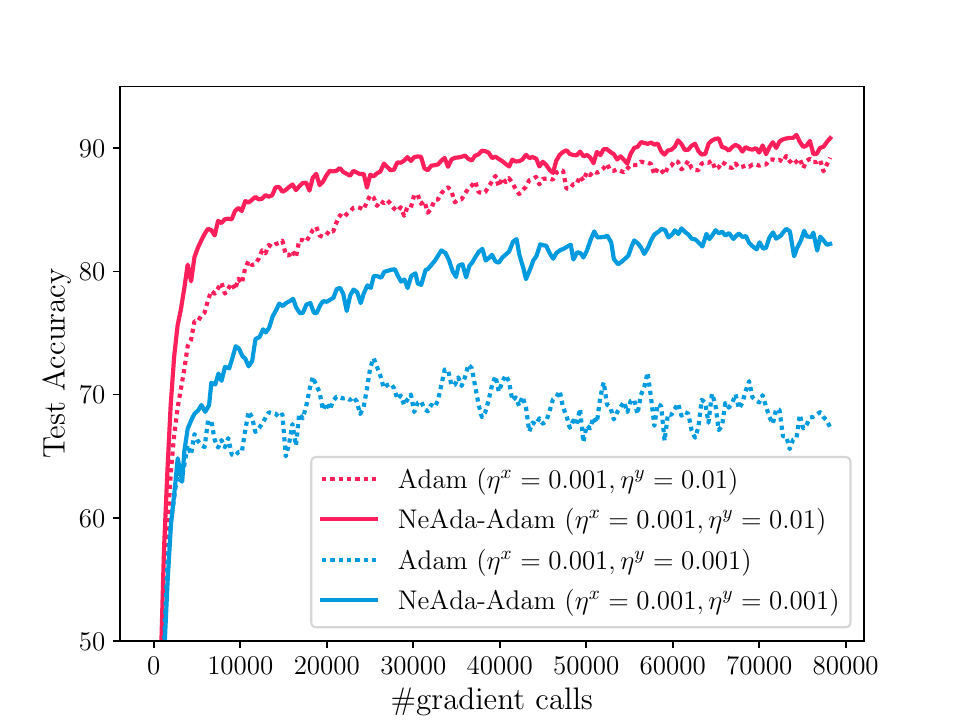}
      \caption{$\epsilon = 0.1$}
      \label{fig:mnist_result_0.1}
    \end{subfigure}
    \hfill
    \begin{subfigure}[b]{0.245\textwidth}
      \centering
      \includegraphics[width=\textwidth]{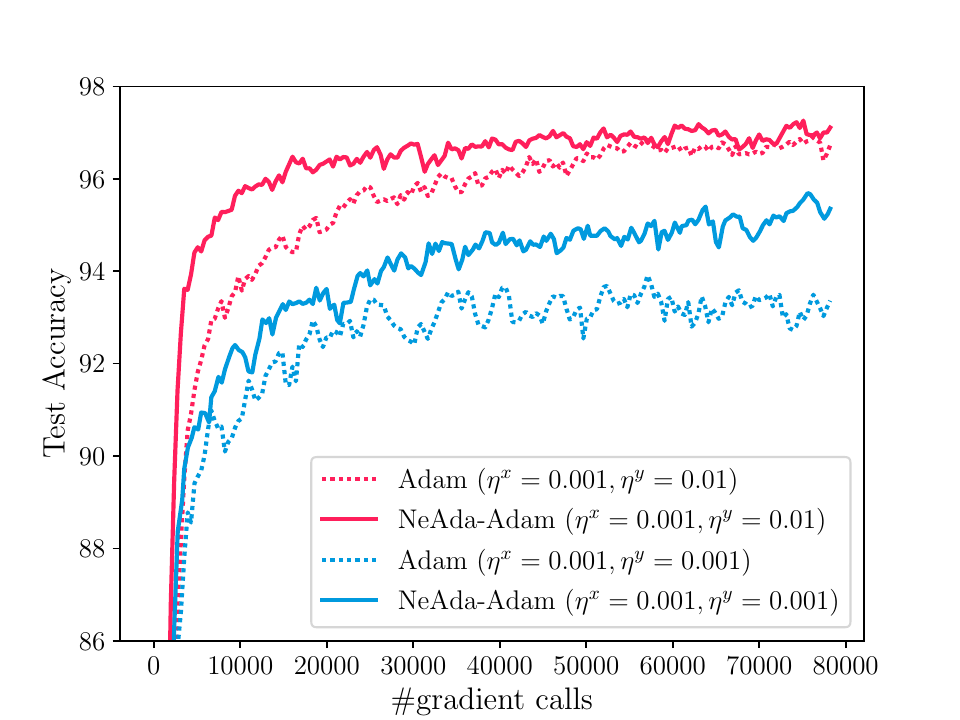}
      \caption{$\epsilon = 0.05$}
      \label{fig:mnist_result_0.05}
    \end{subfigure}
    \hfill
    \begin{subfigure}[b]{0.245\textwidth}
      \centering
      \includegraphics[width=\textwidth]{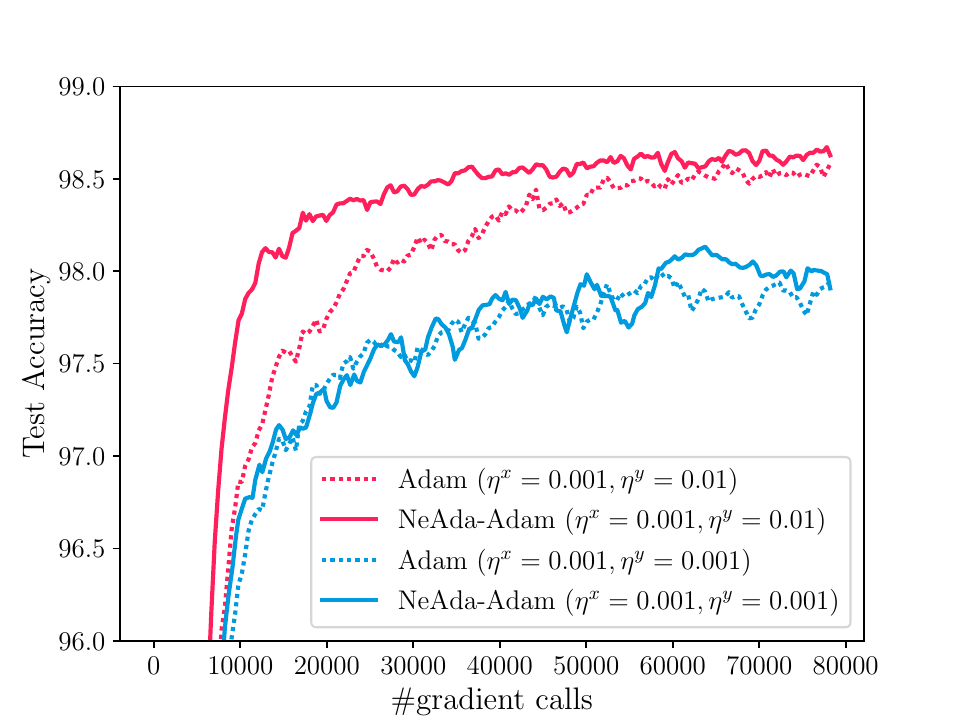}
      \caption{$\epsilon = 0.02$}
      \label{fig:mnist_result_0.02}
    \end{subfigure}
    \begin{subfigure}[b]{0.245\textwidth}
      \centering
      \includegraphics[width=\textwidth]{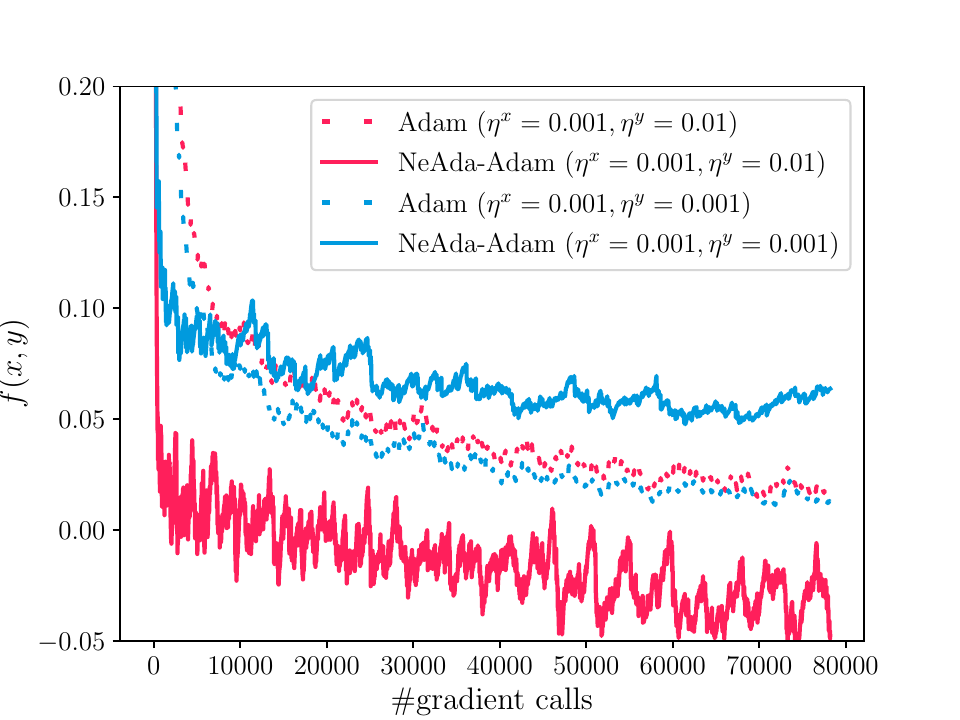}
      \caption{loss}
      \label{fig:mnist_loss}
    \end{subfigure}
    \caption{Results of distributional robustness optimization task on MNIST. $\epsilon$ is the noise level.}%
    \label{fig:mnist_result}
\end{figure}

\section{Conclusion}

Both non-nested and nested adaptive methods are popular in nonconvex minimax problems, e.g., the training of GANs. In this paper, we demonstrate that non-nested algorithms may fail to converge when the time-scale separation is ignorant of the problem parameter even when the objective is strongly-concave in the dual variable with noiseless gradients information. We propose fixes to this problem with a family of nested algorithms--NeAda, that nests the max oracle of the dual variable under an inner loop stopping criterion. The proper stopping criterion will help to balance the outer loop progress and inner loop accuracy. NeAda-AdaGrad attains the near-optimal complexity without a priori knowledge of problem parameters in the nonconvex-strongly-concave setting. It can be a future direction to design parameter-agnostic algorithms for nonconvex-concave minimax problems or more general regimes by leveraging recent progress in nonconvex minimax optimization and the adaptive analysis in this paper. Another interesting direction is to investigate the convergence behavior of Adam-type algorithms with general decaying rates in the strongly convex online optimization.

\section*{Acknowledgement}
This work was supported by an ETH Research Grant funded through the ETH Zurich Foundation.


\bibliographystyle{plainnat}
\bibliography{ref}

\clearpage

\appendix

\section{Helper Lemmas and Proofs for Section \ref{sec::nonconvergence}}
\label{apdx:lemma}

\subsection{Helper Lemmas}

\begin{lemma}[Lemma~4.3 in \citep{lin2020gradient} and Lemma~A.5 in \citep{nouiehed2019solving}]
\label{lemma:smooth}
    Under \Cref{assum:smooth,assum:sc}, define
$\Phi(x)=\max_{y\in\mathcal{Y}} f(x, y)$. Define $y^{*}(x) = \argmax_{y\in\mathcal{Y}} f(x, y)$. Then $y^*(\cdot)$ is $\kappa$-Lipschitz with $\kappa=\frac{l}{\mu}$,  
$\Phi(\cdot)$ is L-smooth with $L:=l+l \kappa$ and $\nabla \Phi(x)=\nabla_{x} f\left(x,
y^{*}(x)\right)$.
\end{lemma}

\bigskip

\begin{lemma}
  \label{lemma:bound_sum}
  Let $x_1, ..., x_T$ be a sequence of non-negative real numbers, $x_1 > 0$ and
  $0 < \alpha < 1$. Then we have 
  \[
    \left(\sum_{t=1}^T x_t\right)^{1-\alpha} 
      \leq \sum_{t=1}^T \frac{x_t}{\left(\sum_{k=1}^t x_k\right)^{\alpha}}
      \leq \frac{1}{1-\alpha} \left(\sum_{t=1}^T x_t\right)^{1-\alpha}.
  \]
  When $\alpha = 1$, we have
  \[
      \sum_{t=1}^T \frac{x_t}{\left(\sum_{k=1}^t x_k\right)^{\alpha}}
      \leq 1 + \log\left(\frac{\sum_{t=1}^t x_t}{x_1}\right).
  \]
\end{lemma}

\begin{remark}
 The case $\alpha = 1/2$ has been noticed in~\citep{auer2002adaptive}, and the upper bound in the case $\alpha = 1$ has already been noticed in \citep{ward2019adagrad}. Here we we extend it to $0 < \alpha \leq 1$.
\end{remark}

\begin{proof}
  For the first inequality, we have
  \begin{align*}
     \sum_{t=1}^T \frac{x_t}{\left(\sum_{k=1}^t x_k\right)^{\alpha}}
     \geq \sum_{t=1}^T \frac{x_t}{\left(\sum_{k=1}^T x_k\right)^{\alpha}}
     = \frac{\sum_{t=1}^T x_t}{\left(\sum_{t=1}^T x_t\right)^{\alpha}}
     = \left(\sum_{t=1}^T x_t\right)^{1-\alpha}.
   \end{align*}
  For the second inequality, we follow a similar procedure as in the proof of
  Lemma~3.5 of~\citep{auer2002adaptive}.
  First consider the case $0 < \alpha < 1$. By Bernoulli's inequality, as $y \leq 1$ and $0 < \alpha < 1$, we have $1 - (1-\alpha)y \geq (1 - y)^{1-\alpha}$. Denoting $S_t = \sum_{k=1}^t x_k$ and $S_0 = 0$,
  by replacing $y$ with $x_t / S_t$, we have
  \[
    (1-\alpha) \frac{x_t}{S_t} \leq 1 - \left(1 - \frac{x_t}{S_t}\right)^{1-\alpha}.
  \]
  Multiplying both sides by $S_t^{1-\alpha}$, then we have
  \[
    (1-\alpha) \frac{x_t}{S_t^\alpha} \leq S_t^{1-\alpha} - S_{t-1}^{1-\alpha}.
  \]
  Summing over the inequalities for $t=1, ..., T$ gives us the desired result.
  For $\alpha=1$, it is proved by \citep{ward2019adagrad}.
\end{proof}

\bigskip

\begin{prop} \label{prop log bound}
If $x^2 \leq (a_1+a_2x)\left(a_3 + a_4\log(a_5a_1 + a_5a_2x) \right)$ with $x, a_1, a_2, a_3, a_4, a_5\geq 0$ and $a_2 > 0$, then 
$$ x \leq \frac{a_1}{a_2} + 16a_2^3a_4^2a_5 + 3a_2^2a_3^2 $$
\end{prop}

\begin{proof}
The proof is similar to Lemma 6 in \citep{li2019convergence}. If $a_2x < a_1$, we have $x \leq a_1/a_2$. Assume $a_2x \geq a_1$, then
$$x^2\leq 2a_2x(a_3 + a_4\log(2a_5a_2x))  \leq 2a_2x \left( a_3 + a_4\sqrt{2a_5a_2x} \right), $$
which implies
$$x \leq 2a_2a_3 + 2a_2a_4\sqrt{2a_5a_2x} \quad \Longrightarrow \quad x^2\leq 8a_2^2a_3^2 + 16a_2^3a_4^2a_5x. $$
Solving this, we get
$$ 
x \leq 8a_2^3a_4^2a_5 + \sqrt{64a_2^6a_4^4a_5^2 + 8a_2^4a_3^4} \leq 16a_2^3a_4^2a_5 + 3a_2^2a_3^2. 
$$

\end{proof}

\begin{prop} \label{prop recursion}
Assume $x_t, a_t, b_t > 0$, for $t = 0, 1, 2,...$, and $x_{t+1} \leq a_tx_t + b_t$, then we have
$$
x_T \leq \left(\prod_{t = 0}^{T-1}a_t\right)x_0 + \sum_{t=0}^{T-2}\left(\prod_{i=t+1}^{T-1}a_i\right)b_t + b_{T-1}, \quad T \geq 2
$$
\end{prop}
\begin{proof}
Let's prove it by induction. It is obvious that this inequality holds for $T=2$:
$$
x_2 = a_1x_1 + b_1 = a_1a_0x_0 + a_1b_0 + b_1.
$$
Assume this inequality holds for $T$, then 
\begin{align*}
    x_{T+1} \leq & a_T\left[\left(\prod_{t = 0}^{T-1}a_t\right)x_0 + \sum_{t=0}^{T-2}\left(\prod_{i=t+1}^{T-1}a_i\right)b_t + b_{T-1} \right] + b_T \\
    = & \left(\prod_{t = 0}^{T}a_t\right)x_0 + \sum_{t=0}^{T-1}\left(\prod_{i=t+1}^{T}a_i\right)b_t + b_{T}.
\end{align*}
\end{proof}

\begin{lemma} \label{prop sum recursion}
Assume $x_t > 0$, for $t = 0, 1, 2,...$, and $x_{t+1} = a_1x_t/(t+1) + a_2/(t+1)$, then we have
$$
\sum_{t=0}^{T}x_t \leq a_2(1+ \log T) + a_2e^{a_1} + x_0e^{a_1}.
$$
\end{lemma}

\begin{proof}
By Proposition \ref{prop recursion}, we have 
\begin{align} \nonumber
    \sum_{t=0}^{T}x_t \leq & x_0 + \left(a_1x_0 +  a_2\right) + \sum_{t=2}^{T} \left[\left(\prod_{i = 0}^{t-1}\frac{a_1}{i+1}\right)x_0 + \sum_{i=0}^{t-2}\left(\prod_{j=i+1}^{t-1}\frac{a_1}{j+1}\right)\frac{a_2}{i+1} + \frac{a_2}{t} \right] \\ \label{sum recursion 1}
    = & x_0 + x_0\sum_{t=1}^T\prod_{i = 0}^{t-1}\frac{a_1}{i+1} + \sum_{t=2}^{T} \left[ \sum_{i=0}^{t-2}\left(\prod_{j=i+1}^{t-1}\frac{a_1}{j+1}\right)\frac{a_2}{i+1}\right] + \sum_{t=1}^{T}\frac{a_2}{t}. 
\end{align}
We note that $\sum_{t=1}^{T}\frac{a_2}{t} \leq a_2(1+ \log T)$ and 
\begin{equation*}
    x_0\sum_{t=1}^T\prod_{i = 0}^{t-1}\frac{a_1}{i+1} = x_0\sum_{t=1}^T\frac{a_1^t}{t!} \leq x_0e^{a_1},
\end{equation*}
where the last inequality can be derived from Taylor expansion of exponential function. Then it remains to bound the third term on the right hand side of (\ref{sum recursion 1}). We can upper bound it by noticing
\begin{align} \nonumber
    \sum_{t=2}^{T} \left[ \sum_{i=0}^{t-2}\left(\prod_{j=i+1}^{t-1}\frac{a_1}{j+1}\right)\frac{a_2}{i+1}\right] = & a_2 \sum_{t=1}^{T-1}\sum_{i=1}^{T-t}\left(\prod_{j=i}^{i+t-1}\frac{a_1}{j+1} \right)\frac{1}{i} \\ \nonumber
    = &  a_2 \sum_{t=1}^{T-1} a_1^t \sum_{i=1}^{T-t}\prod_{j=i}^{i+t}\frac{1}{j} \\ \nonumber
    = &  a_2 \sum_{t=1}^{T-1} a_1^t \sum_{i=1}^{T-t}\frac{1}{t}\left(\prod_{j=i}^{i+t-1}\frac{1}{j} - \prod_{j=i+1}^{i+t}\frac{1}{j} \right) \\ \nonumber
    = & a_2 \sum_{t=1}^{T-1} \frac{a_1^t}{t}\left(\prod_{j=1}^t \frac{1}{j} - \prod_{j= T-t+1}^T \frac{1}{j} \right) \leq a_2 \sum_{t=1}^{T-1} \frac{a_1^t}{t\cdot(t!)} \leq a_2e^{a_1},
\end{align}
where in the third equality we use $\frac{1}{t}\left(\prod_{j=i}^{i+t-1}\frac{1}{j} - \prod_{j=i+1}^{i+t}\frac{1}{j} \right) = \prod_{j=i}^{i+t}\frac{1}{j}$, the last inequality can be derived from Taylor expansion of exponential function; and to see the first equality, the left hand side is the sum of the following
\begin{equation*}
\begin{matrix}
     a_2\times \frac{a_1}{2} & & \\
     a_2\times \frac{a_1}{2} \times \frac{a_1}{3} & \frac{a_2}{2}\times\frac{a_1}{3} & \\
     a_2\times \frac{a_1}{2} \times \frac{a_1}{3} \times \frac{a_1}{4} & \frac{a_2}{2}\times\frac{a_1}{3} \times \frac{a_1}{4} & \frac{a_2}{3} \times \frac{a_1}{4}  \\
    \vdots & \vdots & & \ddots  &  \\
     a_2\times \frac{a_1}{2} \times \dots \times \frac{a_1}{T-1} & \frac{a_2}{2} \times \frac{a_1}{3} \times \dots \times  \frac{a_1}{T-1} & & & \frac{a_2}{T-2}\times \frac{a_1}{T-1} \\
      a_2\times \frac{a_1}{2} \times \dots \times \frac{a_1}{T} & \frac{a_2}{2} \times \frac{a_1}{3} \times \dots \times  \frac{a_1}{T} & & & \frac{a_2}{T-2}\times \frac{a_1}{T-1} \times \frac{a_1}{T} & \frac{a_2}{T-1}\times \frac{a_1}{T},
\end{matrix}
\end{equation*}
and on the right hand side we sum them by each diagonal.

\end{proof}

\subsection{Proofs for Section~\ref{sec::nonconvergence}}

\begin{proof}[Proof for Lemma \ref{lemma:nonconverge}]
Note that $\nabla_x f(x, y) = -L^2x + Ly$ and $\nabla_y f(x, y) = Lx - y$. Then we have
\begin{align*}
    \nabla_x f(x_{t+1}, y_{t+1}) & = -L^2 x_{t+1} + L y_{t+1} \\
   & = -L^2\left[ x_t - \frac{\eta^x}{\sqrt{v_t^x}} m_t^x \right] + L\left[y_t + \frac{r\eta^x}{\sqrt{v_t^y}} m_t^y \right] \\
   & = -L^2 x_{t} + Ly_{t} +  \frac{L^2\eta^x}{\sqrt{v_t^x}}m_t^x + \frac{Lr\eta^x}{\sqrt{v_t^y}} m_t^y . 
\end{align*}
\paragraph{GDA.} With $v_t^x = v_t^y = 1$, $m_t^x = -L^2x_t + Ly_t$ and $m_t^y = Lx_t- y_t$,
\begin{align*}
    \nabla_x f(x_{t+1}, y_{t+1}) & = -L^2 x_{t} + Ly_{t} + L^2\eta^x(-L^2x_t + Ly_t) + Lr\eta^x ( Lx_t- y_t) \\
    & = (-L^2x_t + Ly_t)(1 + L^2\eta^x - r\eta^x) \\
    & = (1 + L^2\eta^x - r\eta^x)\nabla_x f(x_t, y_t).
\end{align*}
\paragraph{Adaptive methods.} Note that $(g_t^x)^2 = L^2(g_t^y)^2$, so by our assumption, $v_t^x = L^2 v_t^y$ for all $t$. Also, with $\beta^x = \beta^y$, we have
\begin{align*}
    m_t^x + rm_t^y & = \beta^xm^x_{t-1} + (1-\beta^x)(-L^2x_t + Ly_t) + r\beta^x m^y_{t-1} + r(1-\beta^x)(Lx_t - y_t) \\
    & = \beta^x(m_{t-1}^x + rm_{t-1}^y) + \left(1-\frac{r}{L}\right)(1-\beta^x)\nabla_x f(x_t, y_t). 
\end{align*}
Recursing this with 
\begin{equation*}
    \nabla_x f(x_{t+1}, y_{t+1}) = \nabla_x f(x_t, y_t) + \frac{L\eta^x}{\sqrt{v_t^y}}(m_t^x + rm_t^y), \text{ and } \quad m_0^x = m_0^y = 0,
\end{equation*}
 when $r \leq L$ we have
$$
  \nabla_x f(x_T, y_T) \geq \nabla_x f(x_0, y_0) \prod_{t=0}^{T-1}\left[ 1 + \frac{L\eta^x}{\sqrt{v_t^x}}(1 - \beta^x)(L - r) \right].
$$
\paragraph{Averaged and best iterate.} We note that the distance from a point $(x, y)$ to the line $y = Lx$, the set of stationary point, is $\frac{|Lx- y|}{\sqrt{L^2 + 1}}$ that is proportional to $|\nabla_x f(x,y)|$. Therefore, the iterate converges to the set of stationary point if and only if the gradient about $x$ converges to $0$. This also explains the best iterate will not converge to the set of stationary point for GDA with $r \leq L^2$ and for adaptive methods with $r \leq L$. 
Average iterate will not converge under the same condition by observing that if an iterate $(x_t, y_t)$ is on the one side of the line $y = Lx$, the next iterate $(x_{t+1}, y_{t+1})$ will stay on the same side. Without loss of generality, assume $(x_t, y_t)$ is on the right of the line $y = Lx$, i.e., $y_t < Lx_t$. By the update of GDA, 
$$
x_{t+1} = x_t + \eta^x(L^2x_t - Ly_t), \quad y_{t+1} = y_t + r\eta^x(Lx_t - y_t),
$$
we have $y_{t+1} < Lx_{t+1}$ as $r \leq L^2$. For adaptive methods, by the recursion of $m_t^x$ and $m_t^y$, if $y_s < Lx_s$ for all $s \leq t$, we have $-m_t^x > Lm_t^y$. The update of adaptive methods can be written as:
$$
x_{t+1} = x_t + \frac{\eta^x}{L\sqrt{v_t^y}}(-m_t^x), \quad y_{t+1} = y_t + \frac{r\eta^x}{\sqrt{v_t^y}}m_t^y.
$$
Then $y_{t+1} < Lx_{t+1}$ as $r \leq L^2$. Now we conclude that the iterate will always stay on the one side of line $y = Lx$.

\end{proof}

\section{Proofs for Section~\ref{sec::convergence}}
\label{apdx:sec3}

\subsection{Proofs for NeAda-AdaGrad}

\paragraph{Proofs for \Cref{lemma stoc}}

\begin{proof}
Part of the proof is motivated by~\citep{ward2019adagrad}. By the smoothness of $\Phi$ from \Cref{lemma:smooth}, we have
\begin{align*}
    \Phi(x_{t+1}) & \leq  \Phi(x_t) + \langle \nabla\Phi(x_t), x_{t+1}-x_t\rangle + \kappa l \|x_{t+1}-x_t\|^2  \\
    & =  \Phi(x_t) - \left\langle \nabla\Phi(x_t), \frac{\eta}{\sqrt{v_{t+1}}} \left( \frac{1}{M}\sum_i\nabla_x f(x_t, y_t; \xi^i_t)\right)\right\rangle + \frac{\kappa l\eta^2}{v_{t+1}} \left\|  \frac{1}{M}\sum_i\nabla_x f(x_t, y_t; \xi^i_t)\right\|^2.
\end{align*}
Note that 
\begin{equation*} 
    \mathbb{E}_{\xi_t}\left[ \frac{\left\langle\nabla\Phi(x_t), \nabla_x f(x_t, y_t) -  \frac{1}{M}\sum_i\nabla_x f(x_t, y_t; \xi^i_t) \right\rangle}{\sqrt{v_t + \|\nabla_x f(x_t, y_t)\|^2 + \sigma^2/M}} \right] = 0.
\end{equation*}
Therefore, 
\begin{align} \nonumber
     & \fakeeq \mathbb{E}_{\xi_t}\left[\frac{\Phi(x_{t+1}) - \Phi(x_t)}{\eta} \right] \\ \nonumber
     & \leq \mathbb{E}_{\xi_t}\left[\left(\frac{1}{\sqrt{v_t + \|\nabla_x f(x_t, y_t)\|^2 + \sigma^2/M}} - \frac{1}{\sqrt{v_{t+1}}}\right) \left\langle\nabla\Phi(x_t),  \frac{1}{M}\sum_i\nabla_x f(x_t, y_t; \xi^i_t) \right\rangle\right] - \\ \label{primal function bound}
     & \fakeeq  \frac{\left\langle\nabla\Phi(x_t), \nabla_x f(x_t, y_t) \right\rangle}{\sqrt{v_t + \|\nabla_x f(x_t, y_t)\|^2 + \sigma^2/M}} + \kappa l \eta  \mathbb{E}_{\xi_t} \left[\frac{\left\|\frac{1}{M}\sum_i\nabla_x f(x_t, y_t; \xi^i_t)\right\|^2}{v_{t+1}}\right].
\end{align}
Now we want to bound the first term on the right hand side and let's denote it as $K$. First we note that
\begin{align*}
    & \fakeeq \left\| \frac{1}{\sqrt{v_t + \|\nabla_x f(x_t, y_t)\|^2 + \sigma^2/M}} - \frac{1}{\sqrt{v_{t+1}}} \right\| \\ 
    & \leq  \left\| \frac{\sqrt{v_{t+1}} - \sqrt{v_t + \|\nabla_x f(x_t, y_t)\|^2 + \sigma^2/M}}{\sqrt{v_{t+1}}\sqrt{v_t + \|\nabla_x f(x_t, y_t)\|^2 + \sigma^2/M}} \right\| \\
    & =  \left\| \frac{\left(\sqrt{v_{t+1}} - \sqrt{v_t + \|\nabla_x f(x_t, y_t)\|^2 + \sigma^2/M}\right)\left(\sqrt{v_{t+1}} + \sqrt{v_t + \|\nabla_x f(x_t, y_t)\|^2 + \sigma^2/M}\right)}{\sqrt{v_{t+1}}\sqrt{v_t + \|\nabla_x f(x_t, y_t)\|^2 + \sigma^2/M}\left(\sqrt{v_{t+1}} + \sqrt{v_t + \|\nabla_x f(x_t, y_t)\|^2 + \sigma^2/M}\right)} \right\| \\
    & = \left\| \frac{\left\|\frac{1}{M}\sum_i\nabla_x f(x_t, y_t; \xi^i_t)\right\|^2 - \|\nabla_x f(x_t, y_t)\|^2 - \sigma^2/M}{\sqrt{v_{t+1}}\sqrt{v_t + \|\nabla_x f(x_t, y_t)\|^2 + \sigma^2/M}\left(\sqrt{v_{t+1}} + \sqrt{v_t + \|\nabla_x f(x_t, y_t)\|^2 + \sigma^2/M}\right)} \right\| \\
    & = \left\| \frac{\left(\left\|\frac{1}{M}\sum_i\nabla_x f(x_t, y_t; \xi^i_t)\right\| - \|\nabla_x f(x_t, y_t)\|\right) \left(\left\|\frac{1}{M}\sum_i\nabla_x f(x_t, y_t; \xi^i_t)\right\| + \|\nabla_x f(x_t, y_t)\|\right) - \sigma^2/M}{\sqrt{v_{t+1}}\sqrt{v_t + \|\nabla_x f(x_t, y_t)\|^2 + \sigma^2/M}\left(\sqrt{v_{t+1}} + \sqrt{v_t + \|\nabla_x f(x_t, y_t)\|^2 + \sigma^2/M}\right)} \right\| \\
    & \leq  \max \bigg\{ \frac{\left| \left\|\frac{1}{M}\sum_i\nabla_x f(x_t, y_t; \xi^i_t)\right\| - \|\nabla_x f(x_t, y_t)\| \right|}{\sqrt{v_{t+1}}\sqrt{v_t + \|\nabla_x f(x_t, y_t)\|^2 + \sigma^2 / M}},  \frac{\sigma/\sqrt{M}}{\sqrt{v_{t+1}}\sqrt{v_t + \|\nabla_x f(x_t, y_t)\|^2 + \sigma^2 / M}} \bigg\}, 
\end{align*}
where in the second equality we use the definition of $v_t$. Therefore we have 
\begin{align} \nonumber
    K & \leq \max \Bigg\{ \mathbb{E}_{\xi_t} \left[ \frac{\left| \left\|\frac{1}{M}\sum_i\nabla_x f(x_t, y_t; \xi^i_t)\right\| - \|\nabla_x f(x_t, y_t)\| \right| \|\nabla \Phi(x_t)\|\left\|\frac{1}{M}\sum_i\nabla_x f(x_t, y_t; \xi^i_t)\right\|^2 }{\sqrt{v_{t+1}}\sqrt{v_t + \|\nabla_x f(x_t, y_t)\|^2 + \sigma^2/M}} \right],  \\ \label{K bound}
    & \fakeeq \mathbb{E}_{\xi_t} \left[ \frac{\frac{\sigma}{\sqrt{M}} \|\nabla \Phi(x_t)\|\left\|\frac{1}{M}\sum_i\nabla_x f(x_t, y_t; \xi^i_t)\right\|^2 }{\sqrt{v_{t+1}}\sqrt{v_t + \|\nabla_x f(x_t, y_t)\|^2 + \sigma^2/M}} \right] \Bigg\}.
\end{align}
 By Young's inequality $ab \leq \frac{1}{4\lambda} a^2 + \lambda b^2$ with $\lambda = \frac{\sigma^2/M}{\sqrt{v_t + \|\nabla_x f(x_t, y_t)\|^2 + \sigma^2/M}}$, $a =  \frac{\left| \left\|\frac{1}{M}\sum_i\nabla_x f(x_t, y_t; \xi^i_t)\right\| - \|\nabla_x f(x_t, y_t)\| \right| \|\nabla \Phi(x_t)\|}{\sqrt{v_t + \|\nabla_x f(x_t, y_t)\|^2 + \sigma^2/M}}$ and $b = \frac{\left\|\frac{1}{M}\sum_i\nabla_x f(x_t, y_t; \xi^i_t)\right\|}{\sqrt{v_{t+1}}} $, the first term on the right hand side of (\ref{K bound}) can be upper bounded by 
 \begin{align*}
    & \fakeeq \mathbb{E}_{\xi_t} \left[\frac{\sqrt{v_t + \|\nabla_x f(x_t, y_t)\|^2 + \sigma^2/M}}{4\sigma^2/M}\left( \frac{\left| \left\|\frac{1}{M}\sum_i\nabla_x f(x_t, y_t; \xi^i_t)\right\| - \|\nabla_x f(x_t, y_t)\| \right| \|\nabla \Phi(x_t)\|}{\sqrt{v_t + \|\nabla_x f(x_t, y_t)\|^2 + \sigma^2/M}} \right)^2  \right] + \\ 
    & \fakeeq \mathbb{E}_{\xi_t} \left[\frac{\sigma^2/M}{\sqrt{v_t + \|\nabla_x f(x_t, y_t)\|^2 + \sigma^2/M}} \left(\frac{\left\|\frac{1}{M}\sum_i\nabla_x f(x_t, y_t; \xi^i_t)\right\|}{\sqrt{v_{t+1}}} \right)^2 \right] \\
    & \leq \frac{\|\nabla \Phi(x_t)\|^2\mathbb{E}_{\xi_t} \left\|\frac{1}{M}\sum_i\nabla_x f(x_t, y_t; \xi^i_t) - \nabla_x f(x_t, y_t)\right\|^2}{\frac{4\sigma^2}{M}\sqrt{v_t + \|\nabla_x f(x_t, y_t)\|^2 + \sigma^2/M}} + \frac{\sigma}{\sqrt{M}}\mathbb{E}_{\xi_t} \left[\frac{\left\|\frac{1}{M}\sum_i\nabla_x f(x_t, y_t; \xi^i_t)\right\|^2}{v_{t+1}} \right] \\ 
    & \leq  \frac{\|\nabla \Phi(x_t)\|^2}{4\sqrt{v_t + \|\nabla_x f(x_t, y_t)\|^2 + \sigma^2/M}} + \frac{\sigma}{\sqrt{M}}\mathbb{E}_{\xi_t} \left[\frac{\left\|\frac{1}{M}\sum_i\nabla_x f(x_t, y_t; \xi^i_t)\right\|^2}{v_{t+1}} \right].  
 \end{align*}
 Similarly, by Young's Inequality with $\lambda = \frac{\sigma^2/M}{\sqrt{v_t + \|\nabla_x f(x_t, y_t)\|^2 + \sigma^2/M}}$, $a =  \frac{ \frac{\sigma}{\sqrt{M}} \|\nabla \Phi(x_t)\|}{\sqrt{v_t + \|\nabla_x f(x_t, y_t)\|^2 + \sigma^2/M}}$ and $b = \frac{\left\|\frac{1}{M}\sum_i\nabla_x f(x_t, y_t; \xi^i_t)\right\|}{\sqrt{v_{t+1}}} $, the second term on the right hand side of (\ref{K bound}) can be upper bounded by 
  \begin{align*}
    & \fakeeq \mathbb{E}_{\xi_t} \left[\frac{\sqrt{v_t + \|\nabla_x f(x_t, y_t)\|^2 + \sigma^2/M}}{4\sigma^2/M}\left( \frac{\frac{\sigma}{\sqrt{M}}\|\nabla \Phi(x_t)\|}{\sqrt{v_t + \|\nabla_x f(x_t, y_t)\|^2 + \sigma^2/M}} \right)^2  \right] + \\ 
    & \fakeeq \mathbb{E}_{\xi_t} \left[\frac{1}{\sqrt{v_t + \|\nabla_x f(x_t, y_t)\|^2 + \sigma^2/M}} \left(\frac{\left\|\frac{1}{M}\sum_i\nabla_x f(x_t, y_t; \xi^i_t)\right\|}{\sqrt{v_{t+1}}} \right)^2 \right] \\
    & \leq  \frac{\|\nabla \Phi(x_t)\|^2}{4\sqrt{v_t + \|\nabla_x f(x_t, y_t)\|^2 + \sigma^2/M}} + \frac{\sigma}{\sqrt{M}}\mathbb{E}_{\xi_t} \left[\frac{\left\|\frac{1}{M}\sum_i\nabla_x f(x_t, y_t; \xi^i_t)\right\|^2}{v_{t+1}} \right].  
 \end{align*}
 Therefore, 
 \begin{equation*}
     K \leq \frac{\|\nabla \Phi(x_t)\|^2}{4\sqrt{v_t + \|\nabla_x f(x_t, y_t)\|^2 + \sigma^2/M}} + \frac{\sigma}{\sqrt{M}}\mathbb{E}_{\xi_t} \left[\frac{\left\|\frac{1}{M}\sum_i\nabla_x f(x_t, y_t; \xi^i_t)\right\|^2}{v_{t+1}} \right].  
 \end{equation*}
Plugging this into (\ref{primal function bound}), 
\begin{align} \nonumber
     & \fakeeq \mathbb{E}_{\xi_t}\left[\frac{\Phi(x_{t+1}) - \Phi(x_t)}{\eta} \right] \\ \nonumber
     & \leq  \frac{\|\nabla \Phi(x_t)\|^2}{4\sqrt{v_t + \|\nabla_x f(x_t, y_t)\|^2 + \sigma^2/M}} + \frac{\sigma}{\sqrt{M}}\mathbb{E}_{\xi_t} \left[\frac{\left\|\frac{1}{M}\sum_i\nabla_x f(x_t, y_t; \xi^i_t)\right\|^2}{v_{t+1}} \right] - \\  \nonumber
     & \fakeeq  \frac{\left\langle\nabla\Phi(x_t), \nabla_x f(x_t, y_t) \right\rangle}{\sqrt{v_t + \|\nabla_x f(x_t, y_t) \|^2 + \sigma^2/M}} + \kappa l \eta  \mathbb{E}_{\xi_t} \left[\frac{\left\|\frac{1}{M}\sum_i\nabla_x f(x_t, y_t; \xi^i_t)\right\|^2}{v_{t+1}}\right] \\ \nonumber
     & \leq  \left( \frac{\sigma}{\sqrt{M}} + \kappa l \eta\right)  \mathbb{E}_{\xi_t} \left[\frac{\left\|\frac{1}{M}\sum_i\nabla_x f(x_t, y_t; \xi^i_t)\right\|^2}{v_{t+1}}\right] 
     - \frac{\|\nabla_x f(x_t, y_t)\|^2}{2\sqrt{v_t + \|\nabla_x f(x_t, y_t)\|^2 + \sigma^2/M}} + \\
     & \fakeeq \frac{\|\nabla_x f(x_t, y_t) - \nabla \Phi(x_t)\|^2 }{2\sqrt{v_t + \|\nabla_x f(x_t, y_t) \|^2 + \sigma^2/M}},
\end{align}
where in the second inequality we use $\|a\|^2/4 - \langle a, b \rangle \leq -\|b\|^2/2 + \|a-b\|^2/2 $. Apply the total law of probability,
\begin{align} \nonumber
    & \fakeeq\frac{1}{2}\sum_{t=0}^{T-1}\mathbb{E}\left[\frac{\|\nabla_x f(x_t, y_t)\|^2}{\sqrt{v_t + \|\nabla_x f(x_t, y_t)\|^2 + \sigma^2/M}} \right] \\ \nonumber
    & \leq  \frac{\Phi(x_0) - \min_x \Phi(x)}{\eta} + \left( \frac{\sigma}{\sqrt{M }} + \kappa l \eta\right) \mathbb{E} \sum_{t=0}^{T-1} \left[\frac{\left\|\frac{1}{M}\sum_i\nabla_x f(x_t, y_t; \xi^i_t)\right\|^2}{v_{t+1}}\right] + \\ \label{grad bound}
    & \fakeeq \mathbb{E} \sum_{t=0}^{T-1} \frac{\|\nabla_x f(x_t, y_t) - \nabla \Phi(x_t)\|^2 }{2\sqrt{v_t + \|\nabla_x f(x_t, y_t) \|^2 + \sigma^2/M}}. 
\end{align}
Denote 
\begin{align*}
    & Z \triangleq \sum_{t=0}^{T-1} \|\nabla_x f(x_t, y_t)\|^2, \quad C \triangleq \sum_{t=0}^{T-1}\mathbb{E}\left[\frac{\|\nabla_x f(x_t, y_t)\|^2}{\sqrt{v_t + \|\nabla_x f(x_t, y_t)\|^2 + \sigma^2/M}} \right], \\
    & D \triangleq \mathbb{E} \sum_{t=0}^{T-1} \left[\frac{\left\|\frac{1}{M}\sum_i\nabla_x f(x_t, y_t; \xi^i_t)\right\|^2}{v_{t+1}}\right], \quad Q \triangleq \mathbb{E} \sum_{t=0}^{T-1} \frac{\|\nabla_x f(x_t, y_t) - \nabla \Phi(x_t)\|^2 }{2\sqrt{v_t + \|\nabla_x f(x_t, y_t) \|^2 + \sigma^2/M}}.
\end{align*}
By \Cref{lemma:bound_sum} with $\alpha=1$,
\begin{align*}
    D \leq & \mathbb{E} \left[1+ \log\left( 1+ \sum_{t=0}^{T-1} \frac{\left\|\frac{1}{M}\sum_i\nabla_x f(x_t, y_t; \xi^i_t)\right\|^2}{v_0}\right) \right] \\
    \leq & 1 + \mathbb{E}\left[\log\left(1 + \frac{\sum_{t=0}^{T-1}\left\|f(x_t, y_t; \xi^i_t)\right\|^2 + \sum_{t=0}^{T-1}\left\|\frac{1}{M}\sum_i\nabla_x f(x_t, y_t; \xi^i_t) - \nabla_x f(x_t, y_t)\right\|^2}{v_0} \right) \right] \\
    \leq & 1 + 2\mathbb{E}\left[\log\left(1 + \frac{Z + \sum_{t=0}^{T-1}\left\|\frac{1}{M}\sum_i\nabla_x f(x_t, y_t; \xi^i_t) - \nabla_x f(x_t, y_t)\right\|^2}{v_0} \right)^{1/2} \right] \\
    \leq & 1 + 2 \mathbb{E}\left[\log\left(1 + \frac{\sqrt{Z}}{\sqrt{v_0}} + \frac{\sqrt{\sum_{t=0}^{T-1}\left\|\frac{1}{M}\sum_i\nabla_x f(x_t, y_t; \xi^i_t) - \nabla_x f(x_t, y_t)\right\|^2}}{\sqrt{v_0}} \right)\right] \\
    \leq & 1 + 2\log\left(1 + \frac{\mathbb{E}[\sqrt{Z}] }{\sqrt{v_0}} +\frac{\mathbb{E}\left[\sqrt{\sum_{t=0}^{T-1}\left\|\frac{1}{M}\sum_i\nabla_x f(x_t, y_t; \xi^i_t) - \nabla_x f(x_t, y_t)\right\|^2}\right]}{\sqrt{v_0}} \right) \\
    \leq & 1 + 2\log\left(1 + \frac{\mathbb{E}[\sqrt{Z}] }{\sqrt{v_0}} +\frac{\sqrt{\sum_{t=0}^{T-1} \sigma^2/M}}{\sqrt{v_0}} \right) \leq 1 + 2\log\left(1 + \frac{\mathbb{E}[\sqrt{Z}] }{\sqrt{v_0}} +\frac{\sqrt{T}\sigma}{\sqrt{v_0M }} \right), 
\end{align*}
where in  the fourth inequality we use $(a+b)^{1/2} \leq a^{1/2} + b^{1/2}$ with $a,b \geq 0$,  the fifth and sixth inequalities are from Jensen's inequality. Also, by $l$-smoothness of $f$,
\begin{equation} \label{q expression}
   Q = \mathbb{E} \sum_{t=0}^{T-1} \frac{\|\nabla_x f(x_t, y_t) - \nabla \Phi(x_t)\|^2 }{2\sqrt{v_t + \|\nabla_x f(x_t, y_t) \|^2 + \sigma^2/M}} \leq  \mathbb{E}\left[ \sum_{t=0}^{T-1} \frac{l^2\|y_t - y^*(x_t)\|^2 }{2\sqrt{v_0}}\right] \triangleq \mathcal{E}.
\end{equation}
 Also,
\begin{align*}
    C \geq & \sum_{t=0}^{T-1}\mathbb{E}\left[\frac{\|\nabla_x f(x_t, y_t)\|^2}{\sqrt{v_0 + \sum_{k = 0}^{T-2} \left\|\frac{1}{M}\sum_i\nabla_x f(x_k, y_k; \xi^i_k)\right\|^2 + \sum_{j=0}^{T-1} \|\nabla_x f(x_j, y_j)\|^2 + \sigma^2/M}} \right] \\
    \geq &  \sum_{t=0}^{T-1}\mathbb{E}\left[\frac{\|\nabla_x f(x_t, y_t)\|^2}{\sqrt{v_0 + 3\sum_{j=0}^{T-1} \|\nabla_x f(x_j, y_j)\|^2 + 2\sum_{k=0}^{T-1} \|\nabla_x f(x_k, y_k) - \frac{1}{M}\sum_i\nabla_x f(x_k, y_k; \xi^i_k)\|^2 + \sigma^2/M}} \right] \\
    \geq & \mathbb{E}\left[\frac{Z}{\sqrt{v_0 + 3Z + 2\sum_{k=0}^{T-1} \|\nabla_x f(x_k, y_k) - \frac{1}{M}\sum_i\nabla_x f(x_k, y_k; \xi^i_k)\|^2 + \sigma^2/M }} \right] \\
    \geq & \frac{\left(\mathbb{E}[\sqrt{Z}]\right)^2}{\mathbb{E}\left[\sqrt{v_0 + 3Z + 2\sum_{k=0}^{T-1} \|\nabla_x f(x_k, y_k) - \frac{1}{M}\sum_i\nabla_x f(x_k, y_k; \xi^i_k)\|^2 + \sigma^2/M }\right]} \\
    \geq & \frac{\left(\mathbb{E}[\sqrt{Z}]\right)^2}{\sqrt{v_0} + 3\mathbb{E}[\sqrt{Z}] + \sigma/\sqrt{M } + 2\sqrt{\sum_{t=1}^{T-1}\sigma^2/M}} \geq \frac{\left(\mathbb{E}[\sqrt{Z}]\right)^2}{\sqrt{v_0} + 3\mathbb{E}[\sqrt{Z}] + 2 \sigma \sqrt{T}/\sqrt{M } },
\end{align*}
where in the fourth inequality we use Holder's inequality, i.e. $\mathbb{E}[X^2]\geq \frac{(\mathbb{E}[XY])^2}{\mathbb{E}[Y^2]}$ with \\ $X = \left(\frac{Z}{\sqrt{v_0 + 3Z + 2\sum_{k=0}^{T-1} \|\nabla_x f(x_k, y_k) - \frac{1}{M}\sum_i\nabla_x f(x_k, y_k; \xi^i_k)\|^2 + \sigma^2/M }} \right)^{1/2}$ and \\ $Y = \left(v_0 + 3Z + 2\sum_{k=0}^{T-1} \|\nabla_x f(x_k, y_k) - \frac{1}{M}\sum_i\nabla_x f(x_k, y_k; \xi^i_k)\|^2 + \sigma^2/M  \right)^{1/4}$, and in the fifth inequality we use $(a+b)^{1/2} \leq a^{1/2} + b^{1/2}$ and Jensen's inequality. Plugging the bounds for $C, D$ and $Q$ into (\ref{grad bound}), 
\begin{align} \nonumber
    & \fakeeq \frac{\left(\mathbb{E}[\sqrt{Z}]\right)^2}{\sqrt{v_0} + 3\mathbb{E}[\sqrt{Z}] + 2 \sigma \sqrt{T}/\sqrt{M } }  \\ \label{Z bound}
    & \leq  \frac{2\left(\Phi(x_0) - \min_x \Phi(x)\right)}{\eta} + \left( \frac{4\sigma}{\sqrt{M }} + 2\kappa l \eta\right)\left[1 + 2\log\left(1 + \frac{\mathbb{E}[\sqrt{Z}] }{\sqrt{v_0}} +\frac{\sigma\sqrt{T}}{\sqrt{v_0}\sqrt{M }} \right)\right] + \mathcal{E}.
\end{align}
Now we want to solve for $\mathbb{E}[\sqrt{Z}]$. Denote $\Delta = \Phi(x_0) - \min_x \Phi(x)$. By Proposition \ref{prop log bound}, we have
\begin{equation*}
     \mathbb{E}[\sqrt{Z}] \leq \frac{\sqrt{v_0}}{3} + \frac{432\Delta^2}{\eta^2} + \frac{2\sigma\sqrt{T}}{3\sqrt{M }} + 432\left(1+\frac{32}{\sqrt{v_0}}\right)\left(\kappa^2l^2\eta^2 + \frac{4\sigma^2}{M} \right) + 108\mathcal{E}^2.
\end{equation*}
We plug this loose upper bound into the logarithmic term on the right hand side of (\ref{Z bound}) and denote the right hand side as $A + \mathcal{E}$. Then we solve the inequality
$$ \frac{\left(\mathbb{E}[\sqrt{Z}]\right)^2}{\sqrt{v_0} + 2\mathbb{E}[\sqrt{Z}] + 2 \sigma \sqrt{T}/\sqrt{M } } \leq A + \mathcal{E},   $$
which gives rise to 
\begin{equation}
    \mathbb{E}[\sqrt{Z}] \leq 2(A + \mathcal{E}) + \left(v_0^{\frac{1}{4}} + 2\sigma^{\frac{1}{2}}T^{\frac{1}{4}}M ^{-\frac{1}{4}} \right)\sqrt{A + \mathcal{E}}.
\end{equation}
Note that 
$$A = \frac{2\Delta}{\eta} + \left( \frac{4\sigma}{\sqrt{M }} + 2\kappa l \eta\right) \left[ 1 + 2\log \left( \mathrm{Poly} \left(T, \mathcal{E}, \frac{\Delta}{\eta}, \frac{\sigma}{\sqrt{M }}, \kappa l\eta, v_0, \frac{1}{v_0} \right) \right) \right].$$

\end{proof}

\paragraph{Proof for \Cref{thm:deter}}

Now we state \Cref{thm:deter} in a more detailed way. 

\begin{theorem} [deterministic] 
Suppose we have a linearly-convergent subroutine $\mathcal{A}$ for maximizing any strongly concave function $h(\cdot)$:
$$\|y^k - y^*\|^2 \leq a_1(1-a_2)^k\|y^0 - y^*\|^2$$
where $y^k$ is $k$-th iterate, $y^*$ is the optimal solution, and $a_1>0$ and $0 < a_2 <1$ are constants that can depend on the parameters of $h$.

Under the same setting as  \Cref{lemma stoc} with $\sigma = 0$,  
for Algorithm \ref{alg::neada-adagrad} with subroutine $\mathcal{A}$ under criterion I: $\|y_t - \text{Proj}_{\mathcal{Y}}(y_t + \nabla_y f(x_t, y_t))\|^2 \leq \frac{1}{t+1}$,  and  $M = 1$,  there exists $t^* \leq  \widetilde{O}\left(\epsilon^{-2} \right)$ such that $(x_{t^*}, y_{t^*})$ is an $\epsilon$-stationary point. Therefore, the total gradient complexity is $\widetilde{O}\left( \epsilon^{-2} \right)$.
\end{theorem}

\begin{proof}
For convenience, we denote $G_y(x, y) = \|y - \text{Proj}_{\mathcal{Y}}(y + \nabla_y f(x, y))\|$ as the gradient mapping about $y$ at $(x, y)$. From Theorem 3.1 in~\citep{pang1987posteriori} and Lemma 10.10 in~\citep{beck2017first}, we have $\frac{\mu}{l+1}\|y - y^*(x)\| \leq \|G_y(x, y)\| \leq (2+l)\|y - y^*(x)\|$. With criterion I, $\mathcal{E}$ can be bounded as the following
$$\mathcal{E}  \leq \mathbb{E}\left[ \sum_{t=0}^{T-1} \frac{l^2(l+1)^2\|G_y f(x_t, y_t)\|^2 }{2\mu^2\sqrt{v_0}}\right] \leq \frac{\kappa^2(l+1)^2}{2\sqrt{v_0}}\sum_{t=0}^{T-1} \frac{1}{t+1} \leq \frac{\kappa^2(l+1)^2(1+\log T)}{2\sqrt{v_0}}, $$
where in the first inequality we use the strong concavity. By setting $\sigma = 0$ and $M = 1$ in  \Cref{lemma stoc}, we have 
\begin{equation*}
    \frac{1}{T}\sum_{t=0}^{T-1} \|\nabla_x f(x_t, y_t)\|^2 \leq \frac{4(A + \mathcal{E})^2}{T} + \frac{\sqrt{v_0}(A + \mathcal{E})}{T} ,
\end{equation*}
where $ A + \mathcal{E} = \widetilde{\mathscr{O}}\left( \frac{\Phi(x_0) - \min_x \Phi(x)}{\eta} +2\sigma + \kappa l\eta + \frac{\kappa^2(l+1)^2}{\sqrt{v_0}} \right)$. We use $\mathscr{O}(\cdot)$ to include the problem parameters in $O(\cdot)$, and similarly $\widetilde{\mathscr{O}}(\cdot)$ ignores the logarithmic terms. Second, we need to compute the inner-loop complexity. At $(t+1)$-th inner loop, we need to bound the initial distance from $y_t$ to the optimal $y$ w.r.t $x_{t+1}$. 
\begin{align*}
    \| y_t - y^*(x_{t+1})\|^2 & \leq  2\|y_t - y^*(x_t)\|^2 + 2 \|y^*(x_t) - y^*(x_{t+1})\|^2 \\
    & \leq  \frac{2(l+1)^2}{\mu^2} \|G_y f(x_t, y_t)\|^2 + 2\kappa^2\|x_t - x_{t+1}\|^2 \\
    & \leq  \frac{2(l+1)^2}{\mu^2}\cdot \frac{1}{t+1} + \frac{2\kappa^2\eta^2}{v_{t+1}} \|\nabla_x f(x_t, y_t) \|^2 \leq \frac{2(l+1)^2}{\mu^2} + 2\kappa^2\eta^2,
\end{align*}
where in the second inequality we use \Cref{lemma:smooth}, and in the third we use $x_{t+1}$ update rule. Therefore  subroutine $\mathcal{A}$ takes $O\left(\frac{1}{a_2}\log(1/t)\right)$ iterations to find $y_{t+1}$ such that $\|G_y(x_{t+1}, y_{t+1})\|^2 \leq (2+l)^2\|y_{t+1} - y^*(x_{t+1})\|^2 \leq \frac{1}{t+2}$. Then we note that 
\begin{align*}
   \sum_{t=0}^{T-1} \|\nabla_x f(x_t, y_t)\|^2 + \|y_t - y^*(x_t)\|^2 & \leq \sum_{t=0}^{T-1} \|\nabla_x f(x_t, y_t)\|^2 + \frac{(l+1)^2}{\mu^2}\|G_y f(x_t, y_t)\|^2 \\
   & \leq 4(A + \mathcal{E})^2 + \sqrt{v_0}(A + \mathcal{E}) + \frac{(l+1)^2}{\mu^2}(1+ \log T).
\end{align*}
So there exists $t \leq \widetilde{\mathscr{O}}\left(\left((A + \mathcal{E})^2+ \sqrt{v_0}(A + \mathcal{E})  + (\kappa^2+1/\mu^2)\right)\epsilon^2\right)$ such that $\|\nabla_x f(x_t, y_t)\|\leq \epsilon$ and $\|y_t - y^*(x_t)\|\leq \epsilon$. Therefore the total complexity is $\widetilde{\mathscr{O}}\left(\left(\frac{(A + \mathcal{E})^2}{a_2} +\frac{ \sqrt{v_0}(A + \mathcal{E})}{a_2}  +\frac{(l+1)^2}{\mu^2a_2}\right)\epsilon^{-2} \right)$ with  $ A + \mathcal{E} = \widetilde{\mathscr{O}}\left( \frac{\Phi(x_0) - \min_x \Phi(x)}{\eta} +2\sigma + \kappa l\eta + \frac{\kappa^2(l+1)^2}{\sqrt{v_0}} \right)$.

\end{proof}

\begin{remark}
As long as we use the stopping criterion $\|y_t - \text{Proj}_{\mathcal{Y}}(y_t + \nabla_y f(x_t, y_t))\|^2 \leq \frac{1}{t+1}$, the exact same oracle complexity as above can be attained for the primal variable, regardless of the subroutine choice. The convergence rate of the subroutine (not necessarily linear rate) will only affect the oracle complexity of the dual variable. 

\end{remark}

\paragraph{Proof for \Cref{thm:stoc}}

Now we state \Cref{thm:stoc} in a more detailed way. Here we consider more general subroutines with $\widetilde{O}(1/k)$ convergence rate. When the subroutine has the convergence rate $O(1/k)$ without additional logarithmic terms, it reduces to the setting of \Cref{thm:stoc}. The proof of the theorem relies on \Cref{prop sum recursion}.

\begin{theorem} [stochastic] 
Suppose we have a sub-linearly-convergent subroutine $\mathcal{A}$ for maximizing any strongly concave function $h(\cdot)$: after $K = k\log^p(k)+1$ iterations
$$   \mathbb{E} \|y^K - y^*\|^2 \leq \frac{b_1\|y^0 - y^*\|^2 + b_2}{k},$$
where $y^k$ is $k$-th iterate, $y^*$ is the optimal solution, $p \in \mathbb{N}$ is an arbitrary non-negative integer and $b_1, b_2>0$ are constants that can depend on the parameters of $h$.

Under the same setting as  \Cref{lemma stoc},  for Algorithm \ref{alg::neada-adagrad} with $M = \epsilon^{-2}$ and subroutine $\mathcal{A}$ under the stopping criterion: at $t$-th inner loop the subroutine stops after $t\log^p(t) + 1$ steps,   there exists $t^* \leq \widetilde{O}\left(\epsilon^{-2} \right)$ such that $(x_{t^*}, y_{t^*})$ is an $\epsilon$-stationary point. Therefore, the total stochastic gradient complexity is 
$\widetilde{O}\left(\epsilon^{-4} \right).$
\end{theorem}

\begin{proof}
First we note that 
$$ \| y_t - y^*(x_{t+1})\|^2 \leq  2\|y_t - y^*(x_t)\|^2 + 2 \|y^*(x_t) - y^*(x_{t+1})\|^2 \leq 2\|y_t - y^*(x_t)\|^2  + 2\kappa^2\eta^2. $$
By the convergence guarantee of subroutine $\mathcal{A}$, after $t\log^p(t) + 1$ inner loop steps, it outputs
\begin{align}
\label{stoc_eb}
    \mathbb{E} \|y_{t+1} - y^*(x_{t+1})\|^2= \frac{b_1\|y_{t} - y^*(x_{t+1})\|^2 + b_2}{t} \leq \frac{2b_1\|y_t - y^*(x_t)\|^2 + 2\kappa^2\eta^2b_1 + b_2}{t}.
\end{align}
Taking expectation of both sides and by \Cref{prop sum recursion}, we have 
\begin{equation} \label{y dist sum}
\mathbb{E}\sum_{t=0}^T\|y_t - y^*(x_t)\|^2 \leq b_3(1+ \log T) + b_3e^{2b_1} + X_0e^{2b_1},
\end{equation}
with $b_3 = 2\kappa^2\eta^2b_1 + b_2$ and $X_0$ denotes $\|y_0 - y^*(x_0)\|^2$. Then
$$
\mathcal{E} = \frac{l^2}{2\sqrt{v_0}}\mathbb{E}\sum_{t=0}^{T-1} \|y_t - y^*(x_t)\|^2 \leq \frac{l^2}{2\sqrt{v_0}} \left[b_3(1+ \log T) + b_3e^{2b_1} + X_0e^{2b_1}\right].
$$
By setting $M = \epsilon^{-2}$ in  \Cref{lemma stoc}, we have 
\begin{equation*}
    \mathbb{E}\left[\sqrt{\frac{1}{T}\sum_{t=0}^{T-1} \|\nabla_x f(x_t, y_t)\|^2} \right] \leq \frac{2(A + \mathcal{E})}{\sqrt{T}} + \frac{v_0^{\frac{1}{4}}\sqrt{A + \mathcal{E}}}{\sqrt{T}} + \frac{2\sqrt{(A + \mathcal{E})\sigma\epsilon}}{T^{\frac{1}{4}}},
\end{equation*}
where $ A = \widetilde{\mathscr{O}}\left( \frac{\Phi(x_0) - \min_x \Phi(x)}{\eta} +        \left(\frac{2\sigma}{\sqrt{M}} + \kappa l\eta \right)(1 + b_1) \right)$. Therefore, 
\begin{align*}
    & \fakeeq \mathbb{E}\left[\sqrt{\frac{1}{T}\sum_{t=0}^{T-1} \|\nabla_x f(x_t, y_t)\|^2} \right] + \sqrt{\mathbb{E}\left[\frac{1}{T}\sum_{t=0}^{T-1} \|y_t - y^*(x_t)\|^2 \right]} \\
     & \leq \frac{2(A + \mathcal{E})}{\sqrt{T}} + \frac{v_0^{\frac{1}{4}}\sqrt{A + \mathcal{E}}}{\sqrt{T}} + \frac{2\sqrt{(A + \mathcal{E})\sigma\epsilon}}{T^{\frac{1}{4}}} + \frac{\sqrt{b_3(1+ \log T) + b_3e^{2b_1} + X_0e^{2b_1}}}{\sqrt{T}}.
\end{align*}
By setting the right hand side to $\epsilon$, we need 
$T = \widetilde{\mathscr{O}}\left( \left((A+\mathcal{E})^2 + \sqrt{v_0}(A+\mathcal{E})(1+\sigma) + b_3 + (b_3+X_0)e^{2b_1} \right)\epsilon^{-2}\right)$ outer loop iterations. Since $M = \epsilon^{-2}$, the sample complexity for $x$ is $T\epsilon^{-2} = \widetilde{O}(\epsilon^{-4})$. Since the inner loop iteration is at most $T\log^p T+1$, the sample complexity for $y$ is $T^2\log^pT + T = \widetilde{O}(\epsilon^{-4})$.

\end{proof}

\begin{remark}
The same sample complexity for the primal variable can be attained as above, as long as (\ref{stoc_eb}) holds. The choice for the subroutine will affect the number of samples needed to achieve (\ref{stoc_eb}), and therefore the sample complexity for the dual variable. Although the complexity above includes an exponential term in $b_1$, we note that $b_1 = 0$ in many subroutines for strongly-convex objectives \citep{NIPS2017_6aed000a, rakhlin2012making, lacoste2012simpler}.
\end{remark}

\subsection{Proofs for Generalized AdaGrad}

\paragraph{Proof of \Cref{thm:gen_ada}}
\begin{proof}
We separate the proof into three parts.
\paragraph{Part I.}
From the update of Algorithm \ref{alg:adagrad_sc}, we have for any $x \in \mathcal{X}$
\begin{align*}
    \|x_{t+1} - x\|^2  =  \left\|x_{t} - \frac{\eta}{v_{t+1}^\alpha}g_t - x\right\|^2 = \|x_t - x\|^2 + \frac{\eta^2}{v_{t+1}^{2\alpha}}\|g_t\|^2 - \frac{2\eta}{v_{t+1}^\alpha}\langle g_t, x_t-x\rangle. 
\end{align*}
Multiple each side by $v_{t+1}^{\alpha}$, 
\begin{align*}
     \fakeeq v_{t+1}^{\alpha}\|x_{t+1} - x\|^2 
     = v_{t+1}^{\alpha}\|x_t - x\|^2 + \frac{\eta^2}{v_{t+1}^\alpha} \|g_t\|^2 - 2\eta \langle g_t, x_t-x\rangle.
\end{align*}
By strong convexity, 
\begin{equation*}
    f_t(x_t) - f_t(x) \leq \langle g_t, x_t - x\rangle - \frac{\mu}{2}\|x_t - x\|^2.
\end{equation*}
Plug it into the previous inequality, 
\begin{align*}
      v_{t+1}^{\alpha }\|x_{t+1} - x\|^2 \leq  v_{t+1}^{\alpha}\|x_t - x\|^2 + \frac{\eta^2}{v_{t+1}^\alpha} \|g_t\|^2 - 
     2\eta[f_t(x_t) - f_t(x^*)] - \eta\mu \|x_t - x\|^2.
\end{align*}
Telescope from $t = 0$ to $T -1 $,
\begin{align} \label{gen_ada_sum}
    2\eta \sum_{t=0}^{T-1} [f_t(x_t) - f_t(x)] \leq v_{1}^{\alpha}\|x_0 - x\|^2 -  v_{T}^{\alpha }\|x_{T} - x\|^2 
    - \sum_{t=1}^{T-1}\left[v_t^{\alpha} - v_{t+1}^{\alpha} + \eta\mu \right]\|x_t - x\|^2 
    +  \sum_{t=0}^{T-1}\frac{\eta^2}{v_{t+1}^\alpha}\|g_t\|^2. 
\end{align}

\paragraph{Part II.} In the part, we focus on the second term on the right hand side of the previous inequality. For convenience, we denote 
$$ B_t =  v_{t+1}^{\alpha} - v_t^{\alpha} - \eta\mu.$$
Denote set $S = \{t: B_t > 0 \}$. We will first bound the number of $t$ for which the coefficient $B_t$ is positive, i.e., $|S|$, for the case $0 < \alpha < 1$. 
We note that
\begin{align} \nonumber
    B_t & = (v_t + \|g_t\|^2)^\alpha - v_t^\alpha - \eta\mu  = v_t^\alpha \left[\left(\frac{v_t + \|g_t\|^2}{v_t} \right)^\alpha - 1 \right] - \eta \mu \\ \label{B bound}
    & \leq v_t^\alpha \left(1 + \alpha \frac{\|g_t\|^2}{v_t} -1 \right) - \eta\mu = \frac{\alpha \|g_t\|^2}{v_t^{1-\alpha}} - \eta \mu, 
\end{align}
where in the inequality we apply Bernoulli's inequality, i.e., $(1+x)^r\leq 1+rx$ with $0\leq r\leq 1$ and $x\geq -1$. If $B_t$ is positive, it leads to 
\begin{align}\label{gt1}
    B_t > 0 \quad &  \Longleftrightarrow \quad \|g_t\|^2 > \frac{\eta\mu}{\alpha}v_t^{1-\alpha} \\ \label{gt2}
    &  \Longrightarrow \quad \|g_t\|^2 > \frac{\eta\mu}{\alpha}v_0^{1-\alpha}
\end{align}
This means $\|g_t\|$ is not small once we observe $B_t > 0$. Since $\|g_t\|^2 \leq G^2$, if the right hand side of (\ref{gt1}) is larger or equal to $G^2$, then $B_t$ can not be positive, i.e. 
\begin{equation*}
   \frac{\eta\mu}{\alpha}v_t^{1-\alpha} \geq G^2 \quad \Longleftrightarrow \quad  v_t \geq \left( \frac{\alpha G^2}{\eta \mu} \right)^{\frac{1}{1-\alpha}}.
\end{equation*}
On the other hand, because $v_{t+1} = v_t + \|g_t\|^2$, (\ref{gt2})  implies that once we observe $B_t > 0$, $v_t$ will increase by at least $\frac{\eta\mu}{\alpha}v_0^{1-\alpha}$. Therefore, it can be positive for only finite times, i.e., 
\begin{equation} \label{s size}
|S| \leq \frac{\left( \frac{\alpha G^2}{\eta \mu} \right)^{\frac{1}{1-\alpha}}}{\frac{\eta\mu}{\alpha}v_0^{1-\alpha}} = \frac{\alpha(\alpha G^2)^{\frac{1}{1-\alpha}}}{(\eta \mu)^{\frac{2-\alpha}{1-\alpha}}v_0^{1-\alpha}}.
\end{equation}
Even when $B_t$ is positive, its value is bounded above from (\ref{B bound}),
\begin{equation} \label{B bound2}
    B_t \leq \frac{\alpha \|g_t\|^2}{v_t^{1-\alpha}} - \eta \mu \leq \frac{\alpha G^2}{v_0^{1-\alpha}}.
\end{equation}
Now it is left to discuss the case $\alpha = 1$. When $\alpha = 1$, 
\begin{equation*}
    B_t = -v_t + v_{t+1} -\eta\mu \leq \|g_t\|^2 - \eta\mu \leq G^2 - \eta\mu.
\end{equation*}
Therefore, when $\eta \geq \frac{G^2}{\mu}$, we have $B_t \leq 0$ for all $t$.

\paragraph{Part III. } In this part we wrap up everything for two cases: i) $0 < \alpha \leq 1$; ii) $\alpha = 1$. From equation (\ref{gen_ada_sum}),
\begin{align} \label{gen_ada_sum2}
    2 \eta \sum_{t=0}^{T-1}[f_t(x_t) - f_t(x)] & \leq v_{1}^{\alpha}\mathcal{D}^2 + \sum_{t \in S} B_t \mathcal{D}^2 +
    \eta^2\sum_{t=0}^{T-1}\frac{1}{v_{t+1}^\alpha}\|g_t\|^2 
\end{align}

\paragraph{Case $0 < \alpha \leq 1$.} By Lemma \ref{lemma:bound_sum}, (\ref{s size}) and (\ref{B bound2}),
\begin{align} \nonumber
    2 \eta \sum_{t=0}^{T-1}[f_t(x_t) - f_t(x)]  
   &  \leq v_{1}^{\alpha}\mathcal{D}^2 + \sum_{t \in S} B_t \mathcal{D}^2 + \frac{\eta^2}{1-\alpha} v_{t+1}^{1-\alpha} \\ \nonumber
   & \leq (v_0 + G^2)^{\alpha}\mathcal{D}^2 + \frac{\alpha(\alpha G^2)^{\frac{2-\alpha}{1-\alpha}}}{(\eta \mu)^{\frac{2-\alpha}{1-\alpha}}v_0^{2-2\alpha}} + \frac{\eta^2}{1-\alpha} v_{t+1}^{1-\alpha}.
\end{align}

\paragraph{Case $ \alpha = 1$.}  We have $B_t \leq 0$ for all $t$ as $\eta \geq \frac{G^2}{\mu}$. Then by Lemma \ref{lemma:bound_sum},
\begin{align*} 
    2\eta\sum_{t=0}^{T-1}[f_t(x_t) - f_t(x)] \leq (v_0 + G^2)\mathcal{D}^2 + \eta^2\log \left(\frac{\sum_{t=0}^{T-1} \|g_t\|^2}{v_0} \right).
\end{align*}
\end{proof}

\begin{remark}
We note that  the regret bounds contain a constant term $\mu^{-\frac{1}{1-\alpha}}$, which increases exponentially as $\alpha$ approaches $1$. However, such term is common even in the convergence result of SGD with a non-adaptive stepsize $\frac{\eta}{t^{\alpha}}$ in strongly-convex stochastic optimization; e.g., Theorem 1 in~\citep{moulines2011non} and Theorem 31 in~\citep{fontaine2021convergence} both contain a term that will not diminish before $\Theta\left(\mu^{-\frac{1}{1-\alpha}} \right)$ iterations. 
\end{remark}

\section{High Probability Convergence Analysis} \label{apdx:high_prob}

We provide a high probability convergence guarantee for
the primal variable of NeAda-AdaGrad~(\Cref{alg::neada-adagrad}).
We make two additional assumptions, which are standard for high probability
analysis~\citep{DBLP:journals/corr/abs-2007-14294,kavis2022high}, one on the
norm-subGaussian~\citep{jin2019short} noise and another on the bounded
gradient.

\begin{assume}[Bounded gradient in $x$] \label{assum:bounded_gradient}
  There exists a constant $G > 0$ such that for any $x$ and $y$, 
  $\|\nabla_x f(x, y)\| \leq G$.
\end{assume}
\begin{assume}[Unbiased norm-subGaussian noise] \label{assum:sub_gaussian}
  $\nabla_x F(x, y; \xi)$ is the unbiased stochastic gradient, and we have 
  \[
  \Ep[\xi]{\exp{\left(\norm{\nabla_x F(x, y; \xi) - \nabla_x
  f(x, y)}^2 / \sigma^2\right)}} \leq \exp(1).
  \]
\end{assume}
\begin{remark}
  To deal with multi-dimensional random variables, norm-subGaussian is
  a common assumption in high probability
  analysis~\citep{DBLP:journals/corr/abs-2007-14294,kavis2022high,jin2021nonconvex,madden2020high}.
  If a random vector is $\sigma$-norm-subGaussian, then it is also $\sigma$-subGaussian (vector) and has the variance bounded by $\sigma^2$.
\end{remark}

\begin{theorem} \label{thm:high_prob}
  Under \Cref{assum:smooth,assum:sc,assum:bounded_gradient,assum:sub_gaussian},
  assume there is a subroutine $\mathcal{A}$ that in the $t$-th outer loop,
  with probability at least $1- \delta$, returns $y_t$ after
  $t+1$ steps and guarantees $\norm*{y_t - y^*(x_t)}^2 \leq O(\log 1/\delta)/(t+1)$.
If we use Algorithm \ref{alg::neada-adagrad} with stopping criterion II,
  then with probability at least $1 - 5\delta$ and $v_0 > 0$, we have
  \begin{gather*}
    \frac{1}{T}\sum_{t=0}^{T-1} \norm*{\nabla_x f(x_t, y_t)}^2
    \leq \frac{1}{T}\Bigg[32\left(2 l \kappa \eta + \frac{\Delta}{\eta}\right)^2 
    + 8 \sqrt{v_0} \left(2l\kappa \eta + \frac{\Delta}{\eta}\right)
    + \frac{32 \sigma^2}{M} \log(1/\delta) \\
    \fakeeq + 10c_1 l^2 \left(1 + \log T \right) \log\left( T^2 / \delta \right) \Bigg]
    + \frac{1}{\sqrt{T}}\left[ \frac{8\sqrt{2}\sigma}{\sqrt{M}} \left(2l\kappa \eta + \frac{\Delta}{\eta}\right) \sqrt{c_2 \log \frac{2dT}{\delta}} \right].
  \end{gather*}
  where $\Delta = \max_{0 \leq t \leq T-1} \Phi(x_t) - \Phi^* \leq O\big(\log (T) \log (T/\delta)\big)$, $d$ is
  dimension of $x_t$ and $c_1, c_2$ are constants.
\end{theorem}

\begin{remark}
The complexity requirement for the subroutine $\mathcal{A}$ is
$O(1/T)$ with logarithmic terms on $1/\delta$ for the strongly concave subproblem. This can be achieved by
\citep{jain2019making,harvey2019simple}~\footnote{
\citet{jain2019making,harvey2019simple} both assume bounded stochastic
gradient.}, although they both require knowledge of the strong convexity parameter.
\end{remark}
\begin{remark}
The theorem implies an $\widetilde{O}(\epsilon^{-4})$ sample complexity for the
primal variable as long as the stopping criterion is satisfied. We do not
provide the complexity for the dual variable, because it needs case-by-case
study depending on the subroutine under this criterion.
The analysis for this theorem
is motivated by recent progress in high probability bound for AdaGrad in
nonconvex optimization~\citep{kavis2022high}.

\end{remark}

We first present the following helper lemmas for the proof.

\begin{lemma}[Lemma~1 in \citep{DBLP:journals/corr/abs-2007-14294}] 
  \label{lemma:bound_MDS_1}
 Let $Z_{0}, \cdots, Z_{T-1}$ be a martingale difference sequence (MDS) with respect to random vectors $\xi_{0}, \cdots, \xi_{T-1}$ and $Y_{t}$ be a sequence of random variables which is $\sigma\left(\xi_{0}, \cdots, \xi_{t-1}\right)$-measurable. Given that $\mathbb{E}\left[\exp \left(Z_{t}^{2} / Y_{t}^{2}\right) \mid \xi_{0}, \cdots \xi_{t-1}\right] \leq \exp (1)$, for any $\lambda>0$ and $\delta \in(0,1)$ with probability at least $1-\delta$,
$$
\sum_{t=0}^{T-1} Z_{t} \leq \frac{3}{4} \lambda \sum_{t=0}^{T-1} Y_{t}^{2}+\frac{1}{\lambda} \log (1 / \delta).
$$
\end{lemma}

\begin{remark} \label{remark:lemma_MDS_1}
In \Cref{thm:high_prob}, we consider mini-batch noise, i.e., in the $t$-th
step, we sample $M$ i.i.d. random noises, and the noise
$\xi_{t}^i$ satisfies \Cref{assum:sub_gaussian} for $i=1,\dots,M$. For ease of exposition of our proof, we note that \Cref{lemma:bound_MDS_1} implies the following result.
Assume $Z_0^0,\cdots,Z_{T-1}^M$ is a martingale difference sequence with respect to $\xi_0^0, \cdots, \xi_{T-1}^M$ (the order within a mini-batch $\{Z_t^{i}\}_{i=1}^{M}$ can be arbitrary),  $\widetilde Y_t$ is $\sigma\left(\xi_0^0, \cdots, \xi_{t-1}^M\right)$ 
measurable,
 and $\mathbb{E}\left[\exp \left((Z_{t}^i)^{2} / \widetilde Y_{t}^{2}\right) \mid \xi_{0}^0, \cdots, \xi_{t-1}^M\right] \leq \exp (1)$. 
Then denoting $\widetilde Z_t \coloneqq \frac{1}{M} \sum_{i=1}^M Z_t^i$,  
with probability at least $1-\delta$, we have
\begin{equation} \label{eq:lemma_MDS_1}
\sum_{t=0}^{T-1} \widetilde Z_{t} \leq \frac{3}{4} \lambda
\sum_{t=0}^{T-1} \widetilde Y_{t}^{2} + \frac{1}{\lambda M} \log (1 / \delta).
\end{equation}
\end{remark}

\begin{lemma}[Corollary~7 in \citep{jin2019short}] \label{lemma:nsgd_hoeffding}
 Let random vectors $X_{1}, \ldots, X_{n} \in
 \mathbb{R}^{d}$, and corresponding filtrations
 $\mathcal{F}_{i}=\sigma\left(X_{1}, \ldots, X_{i}\right)$
 for $i \in[n]$ satisfy that $X_{i} \mid \mathcal{F}_{i-1}$ is
 zero-mean $\sigma_{i}$-norm-subGaussian with $\sigma_{i} \in
 \mathcal{F}_{i-1}$. i.e., $$ \mathbb{E}\left[X_{i} \mid
 \mathcal{F}_{i-1}\right]=0, \quad
 \mathbb{P}\left(\left\|X_{i}\right\| \geq t \mid
 \mathcal{F}_{i-1}\right) \leq 2 e^{-\frac{t^{2}}{2 \sigma_{i}^{2}}}, \quad
 \forall t \in \mathbb{R}, \forall i \in[n]. $$
 There exists an absolute constant $c$ such that for any $\delta>0$, with
 probability at least $1-\delta$ : $$ \left\|\sum_{i=1}^{n} X_{i}\right\| \leq
 c \cdot \sqrt{\sum_{i=1}^{n} \sigma_{i}^{2} \log \frac{2 d}{\delta}}. $$
\end{lemma}

\begin{lemma} \label{lemma:bound_MDS_3}
Under \Cref{assum:sub_gaussian}. For $\delta \in(0,1)$, with probability at least $1-\delta$,
\[
\max_{0 \leq t \leq T-1} \norm*{\frac{1}{M}\sum_{i=1}^{M} \nabla_x F(x_t, y_t; \xi_t^i) - \nabla_x f(x_t, y_t)}^2
\leq \frac{c\sigma^2}{M} \log \left(\frac{2dT}{\delta}\right),
\]
where $c$ is an absolute constant and $d$ is the dimension of $x_t$.
\end{lemma}
\begin{proof} Firstly, using \Cref{lemma:nsgd_hoeffding}, we have the probability for
  \begin{align*}
    \norm*{\frac{1}{M}\sum_{i=1}^{M} \nabla_x F(x_t, y_t; \xi_t^i) - \nabla_x f(x_t, y_t)}^2 \leq k
  \end{align*}
  is at least $1-2d\exp{\left(-\frac{kM}{c_0^2\sigma^2}\right)}$ for some absolute
  constant $c_0$. Then we have
  \begin{align*}
  &\fakeeq \Prob{\max_{0 \leq t \leq T-1} \norm*{\frac{1}{M}\sum_{i=1}^{M} \nabla_x F(x_t, y_t; \xi_t^i) - \nabla_x f(x_t, y_t)}^2
    > k} \\
  &= \Prob{\text{there exists $0 \leq t \leq T-1$, s.t. } \norm*{\frac{1}{M}\sum_{i=1}^{M} \nabla_x F(x_t, y_t; \xi_t^i) - \nabla_x f(x_t, y_t)}^2
    > k} \\
  &\leq \sum_{t=0}^{T-1} \Prob{\norm*{\frac{1}{M}\sum_{i=1}^{M} \nabla_x F(x_t, y_t; \xi_t^i) - \nabla_x f(x_t, y_t)}^2 > k} \\
  &\leq \sum_{t=0}^{T-1} 2d \exp{\left(-\frac{kM}{c_0^2\sigma^2}\right)} = 2dT\exp{\left(-\frac{kM}{c_0^2\sigma^2}\right)}.
  \end{align*}
  Letting $k$ = $\frac{c_0^2 \sigma^2}{M} \log \left( \frac{2dT}{\delta}\right)$
  gives us the desired result.
\end{proof}

\begin{proof}[Proof of \Cref{thm:high_prob}]
  \renewcommand*{\thetermcounter}{\text{\Alph{termcounter}}} 

  Using the inner loop algorithm we assumed, with probability at least
  $1 - \delta/(t+1)^2$, we have $\norm*{y_t - y^*(x_t)}^2 \leq \frac{c_1\log\left((t+1)^2/\delta\right)}{t+1}$,
  where $c_1$ is a constant. Then
  \begin{align*}
    &\fakeeq \Prob{\norm*{y_t - y^*(x_t)}^2 \leq \frac{c_1\log\left((t+1)^2/\delta\right)}{t+1} \quad \text{for all $t=0,\dots,T-1$}} \\
    &\geq 1 - \sum_{t=0}^{T-1} \Prob{\norm*{y_t - y^*(x_t)}^2 > \frac{c_1\log\left((t+1)^2/\delta\right)}{t+1}} \\
    &\geq 1 - \delta \sum_{t=0}^{T-1} \frac{1}{(t+1)^2} \\
    &\geq 1 - 2\delta.
  \end{align*}

  We will use $\norm*{y_t - y^*(x_t)}^2 \leq \frac{c_1\log\left((t+1)^2/\delta\right)}{t+1}$
  for $t=0,\dots,T-1$ throughout the proof.
  For the simplicity of notion, we denote the stochastic gradient as $\nabla_x
  \widetilde f(x_t, y_t) \coloneqq \frac{1}{M}\sum_{i=1}^{M}\nabla_x F(x_t, y_t;
  \xi_t^i)$.
  We start by the smoothness of the primal function. According to \Cref{lemma:smooth},
  $\Phi(x)$ is smooth with parameter $l+l\kappa \leq 2l\kappa$, and
  \begin{align*}
    \Phi(x_{t+1}) - \Phi(x_t) 
    &\leq -\frac{\eta}{\sqrt{v_{t+1}}} \inp*{\nabla_x \widetilde f(x_t, y_t)}{\nabla \Phi(x_t)}
    + \frac{\eta^2 l\kappa}{v_{t+1}} \norm*{\nabla_x \widetilde f(x_t, y_t)}^2 \\
    &= -\frac{\eta}{\sqrt{v_{t+1}}} \norm*{\nabla_x f(x_t, y_t)}^2 
    + \frac{\eta^2l\kappa}{v_{t+1}} \norm*{\nabla_x \widetilde f(x_t, y_t)}^2 \\
    & \fakeeq + \frac{\eta}{\sqrt{v_{t+1}}} \inp*{\nabla_x f(x_t, y_t) - \nabla_x \widetilde f(x_t, y_t)}{\nabla_x f(x_t, y_t)} \\
    & \fakeeq + \frac{\eta}{\sqrt{v_{t+1}}} \inp*{\nabla_x f(x_t, y_t) - \nabla \Phi(x_t)}{\nabla_x \widetilde f(x_t, y_t) - \nabla_x f(x_t, y_t)} \\
    & \fakeeq + \frac{\eta}{\sqrt{v_{t+1}}}\inp*{\nabla_x f(x_t, y_t) - \nabla \Phi(x_t)}{\nabla_x f(x_t, y_t)}.
  \end{align*}
  Multiplying both side by $\frac{\sqrt{v_{t+1}}}{\eta}$ and telescoping through $t=0,\dots,T-1$, we have
  \begin{equation}
  \begin{aligned}
  \sum_{t=0}^{T-1} \norm*{\nabla_x f(x_t, y_t)}^2
  &\leq \underbrace{\sum_{t=0}^{T-1} \frac{\sqrt{v_{t+1}}}{\eta}\left(\Phi(x_t) - \Phi(x_{t+1}) \right)}_{\labelterm{term:5}}
  + \underbrace{\sum_{t=0}^{T-1}\frac{l\kappa \eta}{\sqrt{v_{t+1}}} \norm*{\nabla_x \widetilde f(x_t, y_t)}^2}_{\labelterm{term:1}} \\
  &\fakeeq + \underbrace{\sum_{t=0}^{T-1} \inp*{\nabla_x f(x_t, y_t) - \nabla_x \widetilde f(x_t, y_t)}{\nabla_x f(x_t, y_t)}}_{\labelterm{term:2}} \\
  &\fakeeq + \underbrace{\sum_{t=0}^{T-1} \inp*{\nabla_x f(x_t, y_t) - \nabla \Phi(x_t)}{\nabla_x \widetilde f(x_t, y_t) - \nabla_x f(x_t, y_t)}}_{\labelterm{term:3}} \\
  &\fakeeq + \underbrace{\sum_{t=0}^{T-1} \inp*{\nabla_x f(x_t, y_t) - \nabla \Phi(x_t)}{\nabla_x f(x_t, y_t)}}_{\labelterm{term:4}}.
  \end{aligned} \label{eq:A}
  \end{equation}
  There are 5 terms to bound:
  \begin{enumerate}[wide, labelwidth=!, labelindent=0pt]
    \item \Refterm{term:5}:
      \renewcommand*{\thetermcounter}{\text{\roman{termcounter}}}
      \newcounter{temptermcounter}
      \setcounter{temptermcounter}{\value{termcounter}}
      \setcounter{termcounter}{0}

      Firstly, we will bound $\Delta$. Denote $\Delta_t = \Phi(x_t) - \min_x \Phi(x)$.
      By smoothness of $\Phi(\cdot)$
      and telescoping
      \begin{align*}
        &\fakeeq \Delta_{T} - \Delta_0  \\
        &\leq \underbrace{-\sum_{t=0}^{T-1} \frac{\eta}{\sqrt{v_{t+1}}}\norm*{\nabla_x f(x_t, y_t)}^2}_{\labelterm{term:6}}
        + \underbrace{\sum_{t=0}^{T-1} \frac{\eta}{\sqrt{v_{t+1}}}\inp*{\nabla_x f(x_t, y_t) - \nabla_x \widetilde f(x_t, y_t)}{\nabla_x f(x_t, y_t)}}_{\labelterm{term:7}} \\
        &\fakeeq + \underbrace{\sum_{t=0}^{T-1} \frac{l\kappa\eta^2}{v_{t+1}} \norm*{\nabla_x \widetilde f(x_t, y_t)}^2}_{\labelterm{term:8}}
        + \underbrace{\sum_{t=0}^{T-1} \frac{\eta}{\sqrt{v_{t+1}}} \inp*{\nabla_x f(x_t, y_t) - \nabla \Phi(x_t)}{\nabla_x \widetilde f(x_t, y_t)}}_{\labelterm{term:9}}.
      \end{align*}

      \paragraph{\Refterm{term:7}} We can bound this term by
      \begin{align*}
        \text{(\ref{term:7})}
        &= \sum_{t=0}^{T-1} \frac{\eta}{\sqrt{v_t}} \inp*{\nabla_x f(x_t, y_t) - \nabla_x \widetilde f(x_t, y_t)}{\nabla_x f(x_t, y_t)} \\
        & \fakeeq + \sum_{t=0}^{T-1} \left( \frac{\eta}{\sqrt{v_{t+1}}} - \frac{\eta}{\sqrt{v_t}}\right) \inp*{\nabla_x f(x_t, y_t) - \nabla_x \widetilde f(x_t, y_t)}{\nabla_x f(x_t, y_t)} \\
        &\leq \sum_{t=0}^{T-1} \frac{\eta}{\sqrt{v_t}} \inp*{\nabla_x f(x_t, y_t) - \nabla_x \widetilde f(x_t, y_t)}{\nabla_x f(x_t, y_t)} \\
        & \fakeeq + \frac{1}{2} \sum_{t=0}^{T-1} \left( \norm*{\nabla_x f(x_t, y_t) - \nabla_x \widetilde f(x_t, y_t)}^2 + \norm*{\nabla_x f(x_t, y_t)}^2 \right) \left(\frac{\eta}{\sqrt{v_t}} -\frac{\eta}{\sqrt{v_{t+1}}}\right) \\
        &\leq \underbrace{\sum_{t=0}^{T-1} \frac{\eta}{\sqrt{v_t}} \inp*{\nabla_x f(x_t, y_t) - \nabla_x \widetilde f(x_t, y_t)}{\nabla_x f(x_t, y_t)}}_{\labelterm{term:10}} \\
        & \fakeeq + \underbrace{\frac{\eta}{2\sqrt{v_0}} \left(G^2 + \max_{t=0,\dots,T-1}\norm*{\nabla_x f(x_t, y_t) - \nabla_x \widetilde f(x_t, y_t)}^2\right)}_{\labelterm{term:11}},
      \end{align*}
      where the first inequality is by Cauchy-Schwarz and Young's inequality.
      Note that $\sqrt{v_{t+1}} \geq \sqrt{v_t}$.
      The second inequality is by bounded gradient and telescoping the sum.
      The \refterm{term:10} above can be bounded by \Cref{eq:lemma_MDS_1} with
      $\widetilde Z_t = \frac{\eta}{\sqrt{v_t}} \inp*{\nabla_x f(x_t, y_t) - \nabla_x \widetilde f(x_t, y_t)}{\nabla_x f(x_t, y_t)}$
      and $\widetilde Y_t^2 = \frac{\eta^2\sigma^2}{v_t}\norm*{\nabla_x f(x_t, y_t)}^2$.
      $\{Z_t^i \coloneqq \frac{\eta}{\sqrt{v_t}} \inp*{\nabla_x f(x_t, y_t) -
      \nabla_x F(x_t, y_t; \xi_t^i)}{\nabla_x f(x_t, y_t)}
      \}_{t=0,\dots,T-1}^{i=1,\dots,M}$ is a martingale difference sequence (the order within a mini-batch can be arbitrary) as
      \begin{align*}
        \Ep{Z_t^i \mid \mathcal{F_{\text{before}}} } = 0,
      \end{align*}
      \begin{align*}
        \Ep{\left| Z_t^i \right| \mid \mathcal{F_{\text{before}}} }
        &\leq \Ep{\frac{\eta}{2\sqrt{v_t}} \left( \norm*{\nabla_x f(x_t, y_t) - \nabla_x F(x_t, y_t; \xi_t^i)}^2
        + \norm*{\nabla_x f(x_t, y_t)}^2 \right) \mid \mathcal{F_{\text{before}}} } \\
        &\leq \frac{\eta}{2 \sqrt{v_0} } \left( \sigma^2 + G^2 \right) < \infty.
      \end{align*}
      Then with probability at least $1-\delta$,
      $\text{\refterm{term:10}} \leq \frac{3}{4}\lambda\sigma^2 \sum_{t=0}^{T-1} \frac{\eta^2}{v_t}\norm*{\nabla_x f(x_t, y_t)}^2 + \frac{1}{\lambda M}\log(1/\delta)$,
      where $\lambda$ will be determined later.
      The second \refterm{term:11} can be bounded by \Cref{lemma:bound_MDS_3}, that is with probability at
      least $1-\delta$, $\text{\refterm{term:11}} \leq \frac{\eta}{2\sqrt{v_0}} \left(G^2 + \frac{c_2\sigma^2}{M} \log \frac{2dT}{\delta} \right)$
      with an absolute constant $c_2$.

      \paragraph{\Refterm{term:6}$+$(\ref{term:7})} Combining these two terms:
      \begin{align*}
        \text{term }(\ref{term:6})+(\ref{term:7}) 
        &= \frac{3}{4}\lambda\sigma^2 \sum_{t=0}^{T-1} \frac{\eta^2}{v_t}\norm*{\nabla_x f(x_t, y_t)}^2 
        - \sum_{t=0}^{T-1} \frac{\eta}{\sqrt{v_{t+1}}} \norm*{\nabla_x f(x_t, y_t)}^2  \\
        &\fakeeq + \frac{1}{\lambda M}\log(1/\delta) + \frac{\eta}{2\sqrt{v_0}} \left(G^2 + \frac{c_2\sigma^2}{M} \log \frac{2dT}{\delta} \right).
      \end{align*}
      For the first two terms, we have:
      \begin{align*}
        &\fakeeq \frac{3}{4}\lambda\sigma^2 \sum_{t=0}^{T-1} \frac{\eta^2}{v_t}\norm*{\nabla_x f(x_t, y_t)}^2 
        - \sum_{t=0}^{T-1} \frac{\eta}{\sqrt{v_{t+1}}} \norm*{\nabla_x f(x_t, y_t)}^2  \\
        &\leq \frac{3}{4}\lambda\sigma^2 \sum_{t=0}^{T-1} \frac{\eta^2}{v_t}\norm*{\nabla_x f(x_t, y_t)}^2 
        - \sum_{t=0}^{T-1} \frac{\eta \sqrt{v_0}}{v_{t+1}} \norm*{\nabla_x f(x_t, y_t)}^2  \\
        &= \frac{3}{4}\lambda\sigma^2 \sum_{t=0}^{T-1} \frac{\eta^2}{v_t}\norm*{\nabla_x f(x_t, y_t)}^2 
        - \sum_{t=0}^{T-1} \frac{\eta \sqrt{v_0}}{v_t} \norm*{\nabla_x f(x_t, y_t)}^2 \\
        &\fakeeq + \sum_{t=0}^{T-1} \left(\frac{\eta \sqrt{v_0}}{v_t} - \frac{\eta \sqrt{v_0}}{v_{t+1}}\right) \norm*{\nabla_x f(x_t, y_t)}^2  \\
        &\leq \left(\frac{3}{4}\lambda\sigma^2 - \frac{\sqrt{v_0}}{\eta}\right) \sum_{t=0}^{T-1} \frac{\eta^2}{v_t}\norm*{\nabla_x f(x_t, y_t)}^2 
        + \frac{\eta}{\sqrt{v_0}} G^2.
      \end{align*}
      By letting $\lambda = \frac{4\sqrt{v_0}}{3\eta\sigma^2}$, we can get rid of the
      first term. Therefore,
      \begin{align*}
        \text{term }(\ref{term:6})+(\ref{term:7}) \leq \frac{3\eta}{2\sqrt{v_0}} G^2
        + \frac{3\eta\sigma^2}{4M\sqrt{v_0}}\log(1/\delta) + \frac{c_2\eta\sigma^2}{2M\sqrt{v_0}} \log \frac{2dT}{\delta}.
      \end{align*}

      \paragraph{\Refterm{term:8}} We can use \Cref{lemma:bound_sum} with $\alpha=1$ (the second inequality below):
      \begin{align*}
        &\fakeeq \text{\refterm{term:8}} \\
        &\leq l\kappa \eta^2 \left( \frac{v_0}{v_0} + \sum_{t=0}^{T-1} \frac{1}{v_{t+1}}\norm*{\nabla_x \widetilde f(x_t, y_t)}^2\right) \\
        &\leq l\kappa \eta^2 \left( 1 + \log\left(\frac{1}{v_0} \left(v_0 + \sum_{t=0}^{T-1} \norm*{\nabla_x \widetilde f(x_t, y_t)}^2 \right) \right) \right) \\
        &\leq l\kappa \eta^2 \left( 1 + \log\left(v_0 + \sum_{t=0}^{T-1} \norm*{\nabla_x \widetilde f(x_t, y_t)}^2 \right) -\log v_0 \right) \\
        &\leq l\kappa \eta^2 \left( 1 + \log\left(v_0 + 2\sum_{t=0}^{T-1} \norm*{\nabla_x f(x_t, y_t)}^2 
        + 2\sum_{t=0}^{T-1} \norm*{\nabla_x \widetilde f(x_t, y_t) - \nabla_x f(x_t, y_t)}^2\right) - \log v_0 \right) \\
        &\leq l\kappa \eta^2 \left( 1 + \log\left(v_0 + 2G^2T
        + 2T\max_{t=0,\dots,T-1} \norm*{\nabla_x \widetilde f(x_t, y_t) - \nabla_x f(x_t, y_t)}^2\right) - \log v_0 \right) \\
        &\leq l\kappa \eta^2 \left( 1 + \log\left( 1 + \frac{2G^2T}{v_0}
        + \frac{2c_2T \sigma^2}{v_0 M} \log \frac{2dT}{\delta}\right) \right),
      \end{align*}
      where we use $\norm*{x+y}^2 \leq 2\norm*{x}^2 + 2\norm*{y}^2$ in the third inequality,
      and by \Cref{lemma:bound_MDS_3}, with probability $1-\delta$, we have
      the last inequality.

      \paragraph{\Refterm{term:9}} We divide this term into two parts by Young's inequality:
      \begin{align*}
        \text{\refterm{term:9}} 
        &= \sum_{t=0}^{T-1} \inp*{\nabla_x f(x_t, y_t) - \nabla \Phi(x_t)}{\frac{\eta}{\sqrt{v_{t+1}}} \nabla_x \widetilde f(x_t, y_t)} \\
        &\leq \frac{1}{2}\sum_{t=0}^{T-1}\norm*{\nabla_x f(x_t, y_t) - \nabla \Phi(x_t)}^2 + \frac{1}{2} \sum_{t=0}^{T-1}\frac{\eta^2}{v_{t+1}} \norm*{\nabla_x \widetilde f(x_t, y_t)}^2.
      \end{align*}
      The second term can be upper bounded by the same derivation as we bound \refterm{term:8}.
      As for the first term, we have
      \begin{align*}
        \sum_{t=0}^{T-1} \norm*{\nabla_x f(x_t, y_t) - \nabla \Phi(x_t)}^2
        &\leq l^2 \sum_{t=0}^{T-1} \norm*{y_t - y^*(x_t)}^2 \\
        &\leq l^2 \sum_{t=0}^{T-1} \frac{c_1\log\left((t+1)^2/\delta\right)}{t+1} \\
        &\leq l^2 \sum_{t=0}^{T-1} \frac{c_1\log\left(T^2/\delta\right)}{t+1} \\
        &\leq c_1 l^2 \left(1 + \log T \right) \log\left( T^2 / \delta \right).
      \end{align*}

      \renewcommand*{\thetermcounter}{\text{\Alph{termcounter}}}
      \setcounter{termcounter}{\value{temptermcounter}}

      \paragraph{In total} Summarizing the above bounds, we have
      \begin{align*}
        \Delta_{T} - \Delta_0
        &\leq \frac{3\eta}{2\sqrt{v_0}} G^2 + \frac{3\eta\sigma^2}{4M\sqrt{v_0}}\log(1/\delta) 
        + \frac{c_2\eta\sigma^2}{2M\sqrt{v_0}}  \log \frac{2dT}{\delta}
        + \frac{c_1 l^2}{2} \left(1 + \log T \right) \log\left( T^2 / \delta \right) \\
        &\fakeeq  + \left(l\kappa + \frac{1}{2}\right)\eta^2 \left(1 + \log\left( 1 + \frac{2G^2T}{v_0} + \frac{2c_2\sigma^2T}{v_0 M} \log \frac{2dT}{\delta}\right) \right).
      \end{align*}
      And $\Delta$ has the same upper bound as above, which is $O\big(\log (T) \log (T/\delta)\big)$.
      Let us go back to \refterm{term:5}, where we have
      \begin{align*}
        \text{\refterm{term:5}} 
        &= \sum_{t=0}^{T-1} \frac{\sqrt{v_{t+1}}}{\eta} \left( \Delta_t - \Delta_{t+1}\right) \\
        &\leq \frac{\sqrt{v_1}}{\eta}\Delta_0 + \frac{1}{\eta} \sum_{t=1}^{T-1}\Delta_t \left(\sqrt{v_{t+1}} - \sqrt{v_t}\right) \\
        &\leq \frac{\sqrt{v_1}}{\eta}\Delta + \frac{1}{\eta}\Delta \sum_{t=1}^{T-1} \left(\sqrt{v_{t+1}} - \sqrt{v_t}\right) \\
        &= \frac{\sqrt{v_T}}{\eta}\Delta.
      \end{align*}

    \item \Refterm{term:1}:

      We can bound this term by \Cref{lemma:bound_sum} (the second inequality):
      \begin{align*}
        \text{\refterm{term:1}} 
        &\leq l \kappa \eta \left(\frac{v_0}{\sqrt{v_0}} + \sum_{t=0}^{T-1}\frac{1}{\sqrt{v_{t+1}}} \norm*{\nabla_x \widetilde f(x_t, y_t)}^2\right) \\
        &\leq 2 l \kappa \eta \sqrt{v_0 + \sum_{t=0}^{T-1} \norm*{\nabla_x \widetilde f(x_t, y_t)}^2}
        = 2 l \kappa \eta \sqrt{v_T}.
      \end{align*}

    \item \Refterm{term:2}:

      We can apply \Cref{eq:lemma_MDS_1} with $\widetilde Z_t = \inp*{\nabla_x f(x_t, y_t) - \nabla_x \widetilde f(x_t, y_t)}{\nabla_x f(x_t, y_t)}$,
      $\widetilde Y_t^2 = \sigma^2\norm*{\nabla_x f(x_t, y_t)}^2$ and $\lambda = \frac{1}{3\sigma^2}$,
      then with probability $1-\delta$,
      \begin{align*}
        \text{\refterm{term:2}} 
        \leq \frac{1}{4}\sum_{t=0}^{T-1}\norm*{\nabla_x f(x_t, y_t)}^2 + \frac{3\sigma^2}{M}\log(1/\delta),
      \end{align*}
      where the first term can be moved to the LHS of \Cref{eq:A}.

    \item \Refterm{term:3}:

      Using \Cref{eq:lemma_MDS_1} with $\widetilde Z_t = \inp*{\nabla_x f(x_t, y_t) - \nabla \Phi(x_t)}{\nabla_x \widetilde f(x_t, y_t) - \nabla_x f(x_t, y_t)}$,
      $\widetilde Y_t^2 = \sigma^2 \norm*{\nabla_x f(x_t, y_t) - \nabla \Phi(x_t)}^2$
      and $\lambda = 1/\sigma^2$, we have with probability at least
      $1-\delta$,
      \begin{align*}
        \text{\refterm{term:3}}
        &\leq \frac{3}{4} \sum_{t=0}^{T-1} \norm*{\nabla_x f(x_t, y_t) - \nabla \Phi(x_t)}^2
        + \frac{\sigma^2}{M} \log(1/\delta) \\
        &\leq \frac{3l^2}{4} \sum_{t=0}^{T-1} \norm*{y_t - y^*(x_t)}^2
        + \frac{\sigma^2}{M} \log(1/\delta) \\
        &\leq \frac{3l^2}{4} \sum_{t=0}^{T-1} \frac{c_1\log\left((t+1)^2/\delta\right)}{t+1}
        + \frac{\sigma^2}{M} \log(1/\delta) \\
        &\leq \frac{3c_1 l^2}{4} \left(1 + \log T \right) \log\left( T^2 / \delta \right)
        + \frac{\sigma^2}{M} \log(1/\delta).
      \end{align*}

    \item \Refterm{term:4}:

      By Cauchy-Schwarz and Young's inequality, we have
      \begin{align*}
        \text{\refterm{term:4}}
        &\leq \frac{1}{2} \sum_{t=0}^{T-1} \norm*{\nabla_x f(x_t, y_t) - \nabla \Phi(x_t)}^2
        + \frac{1}{2} \sum_{t=0}^{T-1} \norm*{\nabla_x f(x_t, y_t)}^2 \\
        &\leq \frac{c_1 l^2}{2} \left(1 + \log T \right) \log\left( T^2 / \delta \right)
        + \frac{1}{2} \sum_{t=0}^{T-1} \norm*{\nabla_x f(x_t, y_t)}^2,
      \end{align*}
      where the second term can be moved the LHS of \Cref{eq:A}.
  \end{enumerate}

  Summarizing the terms~(\ref{term:5}),~(\ref{term:1}),~(\ref{term:2}),~(\ref{term:3})~and~(\ref{term:4}), we can re-write \Cref{eq:A} as
  \begin{align*}
    \frac{1}{4} \sum_{t=0}^{T-1} \norm*{\nabla_x f(x_t, y_t)}^2
    \leq \left(2l\kappa \eta + \frac{\Delta}{\eta}\right) \sqrt{v_T}
    + \frac{4\sigma^2}{M} \log(1/\delta) + \frac{5}{4}c_1 l^2 \left(1 + \log T \right) \log\left( T^2 / \delta \right).
  \end{align*}
  It remains to handle $\sqrt{v_T}$ in the RHS:
  \begin{align*}
    \sqrt{v_T} 
    &= \sqrt{v_0 + \sum_{t=0}^{T-1}\norm*{\nabla_x \widetilde f(x_t, y_t)}^2} \\
    &= \sqrt{v_0 + \sum_{t=0}^{T-1}\norm*{\nabla_x f(x_t, y_t) + \nabla_x \widetilde f(x_t, y_t) - \nabla_x f(x_t, y_t)}^2} \\
    &\leq \sqrt{v_0 + 2\sum_{t=0}^{T-1}\norm*{\nabla_x f(x_t, y_t)}^2 + 2 \sum_{t=0}^{T-1} \norm*{\nabla_x \widetilde f(x_t, y_t) - \nabla_x f(x_t, y_t)}^2} \\
    &\leq \sqrt{v_0 + 2\sum_{t=0}^{T-1}\norm*{\nabla_x f(x_t, y_t)}^2 + 2 T \max_{t=0,\dots,T-1} \norm*{\nabla_x \widetilde f(x_t, y_t) - \nabla_x f(x_t, y_t)}^2} \\
    &\leq \sqrt{v_0} + \sqrt{2\sum_{t=0}^{T-1}\norm*{\nabla_x f(x_t, y_t)}^2} + \sqrt{2T\max_{t=0,\dots,T-1} \norm*{\nabla_x \widetilde f(x_t, y_t) - \nabla_x f(x_t, y_t)}^2} \\
    &\leq \sqrt{v_0} + \sqrt{2\sum_{t=0}^{T-1}\norm*{\nabla_x f(x_t, y_t)}^2} + \sigma\sqrt{\frac{2c_2T}{M}\log\frac{2dT}{\delta}},
  \end{align*}
  where the last inequality holds by \Cref{lemma:bound_MDS_3}.
  With this, in total,
  \begin{align*}
    \frac{1}{4} \sum_{t=0}^{T-1} \norm*{\nabla_x f(x_t, y_t)}^2
    &\leq \sqrt{2} \left(2l\kappa \eta + \frac{\Delta}{\eta}\right) \sqrt{\sum_{t=0}^{T-1} \norm*{\nabla_x f(x_t, y_t)}^2} \\
    &\fakeeq + \left(2l\kappa \eta + \frac{\Delta}{\eta}\right) \left(\sqrt{v_0} + \sigma\sqrt{\frac{2c_2T}{M}\log\frac{2dT}{\delta}} \right) \\
    &\fakeeq + \frac{4 \sigma^2}{M} \log(1/\delta) + \frac{5}{4}c_1 l^2 \left(1 + \log T \right) \log\left( T^2 / \delta \right).
  \end{align*}
  Regarding this inequality as a quadratic of $\sqrt{\sum_{t=0}^{T-1} \norm*{\nabla_x f(x_t, y_t)}^2}$
  and solving for its positive root, we have
  \begin{align*}
    &\fakeeq \sum_{t=0}^{T-1} \norm*{\nabla_x f(x_t, y_t)}^2  \\
    &\leq \left\{\rule{0cm}{1cm}\right. 2\sqrt{2}\left(2 l \kappa \eta + \frac{\Delta}{\eta}\right)
     + 2 \sqrt{\splitfrac{2\left(2 l \kappa \eta + \frac{\Delta}{\eta}\right)^2
      + \left(2l\kappa \eta + \frac{\Delta}{\eta}\right) \left(\sqrt{v_0} + \sigma\sqrt{\frac{2c_2T}{M}\log\frac{2dT}{\delta}} \right)}{
+ \frac{4 \sigma^2}{M} \log(1/\delta) + \frac{5}{4}c_1 l^2 \left(1 + \log T \right) \log\left( T^2 / \delta \right)}} \left.\rule{0cm}{1cm}\right\}^2 \\
    &\leq 32\left(2 l \kappa \eta + \frac{\Delta}{\eta}\right)^2 
    + 8 \left(2l\kappa \eta + \frac{\Delta}{\eta}\right) \left(\sqrt{v_0} + \sigma\sqrt{\frac{2c_2T}{M}\log\frac{2dT}{\delta}} \right) \\
    &\fakeeq + \frac{32 \sigma^2}{M} \log(1/\delta) + 10 c_1 l^2 \left(1 + \log T \right) \log\left( T^2 / \delta \right),
  \end{align*}
  which gives us
  \begin{gather*}
    \frac{1}{T}\sum_{t=0}^{T-1} \norm*{\nabla_x f(x_t, y_t)}^2
    \leq \frac{1}{T}\Bigg[32\left(2 l \kappa \eta + \frac{\Delta}{\eta}\right)^2 
    + 8 \sqrt{v_0} \left(2l\kappa \eta + \frac{\Delta}{\eta}\right)
    + \frac{32 \sigma^2}{M} \log(1/\delta) \\
    \fakeeq + 10c_1 l^2 \left(1 + \log T \right) \log\left( T^2 / \delta \right) \Bigg]
    + \frac{1}{\sqrt{T}}\left[ \frac{8\sqrt{2}\sigma}{\sqrt{M}} \left(2l\kappa \eta + \frac{\Delta}{\eta}\right) \sqrt{c_2 \log \frac{2dT}{\delta}} \right].
  \end{gather*}

\end{proof}

\end{document}